\documentclass[11pt]{amsart}
\usepackage{amssymb}
\usepackage{mathtools}
\usepackage[hidelinks]{hyperref}
\usepackage{braket}

\newtheorem{theorem}{Theorem}[section]
\newtheorem{lemma}[theorem]{Lemma}
\newtheorem{proposition}[theorem]{Proposition}
\newtheorem{corollary}[theorem]{Corollary}
\theoremstyle{remark}
\newtheorem{remark}[theorem]{Remark}
\newtheorem{example}[theorem]{Example}
\theoremstyle{plain}
\newtheorem{theoremA}{Theorem}

\numberwithin{equation}{section}

\newcommand{\intZ}{\mathbb{Z}}
\newcommand{\ratQ}{\mathbb{Q}}
\newcommand{\ratQp}{\ratQ_p}
\newcommand{\intring}[1]{O_{#1}}
\newcommand{\OK}{\intring{K_b}}
\newcommand{\poldisc}[1]{\operatorname{disc}(#1)}
\newcommand{\disc}[1]{D_{#1}}
\DeclareMathOperator{\Norm}{N}

\begin{document}

\title[Cubic fields with explicit fundamental units]{On an infinite family of
cubic fields with explicit fundamental units}

\author{Iwao Kimura}
\address{Department of Mathematics, University of Toyama,
3190 Gofuku, Toyama 930-8555, Japan}
\email{iwao@sci.u-toyama.ac.jp}

\author{Hikaru Umemoto}
\address{Graduate School of Science and Engineering, University of Toyama,
3190 Gofuku, Toyama 930-8555, Japan}
\email{m25c1008@ems.u-toyama.ac.jp}


\begin{abstract}
For an integer $b\neq0,1$ let $\theta$ be the unique real root of
$f_b(x)=x^3-3bx-b^3$ and let $K_b=\ratQ(\theta)$. We exhibit an explicit set of
$b$ of positive density for which $\eta_b=-1/(\theta-(b+1))$ is the fundamental
unit of $K_b$, and we determine the set of $b$ for which $\eta_b$ is the square of
a unit: it is parametrized by the Pell equation $D^2-3E^2=1$, hence infinite, and
under an explicit mild condition on the field discriminant it accounts for all $b$
for which $\eta_b$ is not the fundamental unit. That condition fails for only four
$b$ with $|b|\le3000$, and at one of them, $b=3$, the conclusion itself fails. In
the order $\intZ[\eta_b]$ generated by $\eta_b$, by contrast, $\eta_b$ is the
fundamental unit for every $b$, so these exceptions are a phenomenon of the
maximal order. As an application we construct infinitely many biquadratic fields
whose $3$-class field tower has length greater than $1$.
\end{abstract}

\maketitle

\bibliographystyle{amsplain}

\section{Introduction}

\subsection{The family and its origin}\label{sec:family}
Let $\intZ$ be the ring of rational integers and $\ratQ$ the field of rational
numbers. For $b\in\intZ$ put
\[
  f_b(x)=x^3-3bx-b^3\in\intZ[x],
\]
whose discriminant is $\poldisc{f_b(x)}=-27\,b^3(b^3-4)$.
Then $\poldisc{f_b(x)}<0$ if and only if $b\neq0,1$, and \emph{throughout the
paper $b$ denotes an integer with $b\neq0,1$.} For such $b$ the field $K_b=\ratQ(\theta)$
generated by the unique real root $\theta$ of $f_b(x)$ is a cubic
field which is not totally real; the number $r_1$ of real embeddings of $K_b$ and
the number $r_2$ of conjugate pairs of complex embeddings are both one. Hence, by Dirichlet's unit
theorem, the unit group $\OK^\times$ of the ring $\OK$ of integers of $K_b$ has
free part of rank one. The two excluded values fail for
different reasons: $f_0(x)=x^3$ is reducible, whereas $f_1(x)=x^3-3x-1$ is
irreducible with $\poldisc{f_1(x)}=81>0$, so $K_1$ is the \emph{totally real} cyclic cubic field
of conductor $9$.

Numerical exploration suggests the following, where
\[
  \epsilon_b\coloneqq\theta-(b+1), \quad
  \text{or equivalently,}\quad \eta_b \coloneqq \frac{-1}{\epsilon_b} = \frac{1}{(b+1)-\theta} .
\]
\begin{enumerate}
\item[\textup{(i)}] $\eta_b$ generates $\OK^\times$ modulo $\pm1$ for infinitely
many $b$; the exceptions are sparse, but infinite in number and governed by a Pell
equation.
\item[\textup{(ii)}] The exceptional $b$ are those for which the index
$[\OK:\intZ[\theta]]$ is anomalously large.
\end{enumerate}
We prove both. The Pell equation governs exactly the $b$ for which $\eta_b$ is a
\emph{square}; we call these the \emph{square-exceptions}. That they are all the
exceptions is proved under an explicit mild
condition on the discriminant $\disc{K_b}$ of $K_b$; the sporadic value $b=3$
shows that some such condition is needed.

The family $f_b(x)$ is a close relative of the family
\begin{equation}\label{eq:sister}
  x^3-3x-b^3,
\end{equation}
introduced by Ishida~\cite{MR959788} in his construction of cubic fields
admitting an unramified cyclic cubic extension. For~\eqref{eq:sister} the element
$\varepsilon=1/(1-b(\theta-b))$ is a unit; Kaneko~\cite{MR2037576,MR3363170}
proved it to be the \emph{fundamental} unit for all but finitely many $b$ and
determined the integral bases, Lee--Spearman~\cite{MR2773123} determined the
finite exceptional set, and Kaneko~\cite{MR3363170} used the family to construct,
via a criterion of Yoshida~\cite{MR1970518,MR2037577}, biquadratic fields whose
$3$-class field tower has length $>1$. The two families are subfamilies of
$x^3-3x-n$: by Remark~\ref{rem:sister} the element $\zeta=(\theta^2-2b)/b$ of $K_b$
satisfies $\zeta^3-3\zeta=b^3-2$, so $f_b(x)$ is the member $n=b^3-2$
and~\eqref{eq:sister} the member $n=b^3$. The Diophantine systems differ: for
$f_b(x)$ the square-unit condition is a Pell equation (Theorem~\ref{thmA:exc}), with
infinitely many solutions, whereas for~\eqref{eq:sister} it has finitely many.

Families in which a fundamental unit is given explicitly in terms of the
parameter have long been studied: real quadratic fields of Richaud--Degert
type~\cite{Degert1958}; the simplest cubic fields of Shanks~\cite{Shanks1974},
generated by a root $\rho$ of $x^3=ax^2+(a+3)x+1$, where $\rho$ and $1+\rho$
are units and the regulator is therefore known at once; and, for quartic
fields, the families of Fleckinger and
V\'erant~\cite{FleckingerVerant1995,Verant1997}, parametrized by torsion
points of an elliptic curve, where the ring of integers is determined along
with the units, as it is here. The family of~\cite{Verant1997} is the closest
analogue of ours: it is totally imaginary, hence of unit rank one; it is
infinite because $t^2-2^6$ is squarefree for infinitely many $t$; and its
generating root is proved to be a fundamental unit by playing an upper bound
for that root against a lower bound for the regulator, which is the argument
of \S\ref{sec:fund} below with a different lower bound in place of Artin's
inequality. Cutting across such families is the question of when a unit $u$ is
fundamental in the order $\intZ[u]$ that it generates. For unit rank one that
question has a complete answer, given by
Louboutin~\cite[Thm.~8]{Louboutin2015} \textup(the cubic case goes back to
Nagell~\cite{Nagell1930}\textup): one infinite family of exceptions, and eight
sporadic ones. Our family meets neither
\textup(Remark~\ref{rem:order}\textup), so $\eta_b$ is the fundamental unit of
$\intZ[\eta_b]$ for every $b$. The exceptions of Theorem~\ref{thmA:exc} are of
a different kind: they occur in the maximal order, in which $\intZ[\theta]$
has index $[\OK:\intZ[\theta]]>1$ for every $b\neq-1$.

\subsection{Main results}\label{sec:results}
Recall that $\theta$ denotes the unique real root of $f_b(x)$, and let $v_p$ denote the
$p$-adic additive valuation, normalized by $v_p(p)=1$.

We first determine the ring $\OK$ of integers. Since the discriminant is invariant
under translation, this amounts to computing the single index
$a\coloneqq[\OK:\intZ[\theta]]$, which we do prime by prime
(Theorems~\ref{thm:idxprime} and~\ref{thm:idx23}): away from $6$ one has
$v_p(a)=\lfloor\tfrac32v_p(b)\rfloor$ for $p\mid b$ and
$v_p(a)=\lfloor\tfrac12v_p(b^3-4)\rfloor$ for $p\nmid b$, and $v_2(a)$, $v_3(a)$
are given by explicit formulas as well. For squarefree $b$ with $3\nmid b$ and
$b^3-4$ squarefree this yields $a=|b|$, the integral basis
$\bigl\{1,\theta,(\theta^2-3b)/b\bigr\}$ and $\disc{K_b}=-27b(b^3-4)$
(Corollary~\ref{cor:idxmain}).

Unlike Ishida--Kaneko's family~\eqref{eq:sister}, where $2$ never divides the
index, here both $2$ and $3$ can divide it (Theorem~\ref{thm:idx23}). The prime
$3$ is totally and wildly ramified in $K_b$ when $3\nmid b$; when $3\mid b$ it is
unramified for $v_3(b)$ odd, and tamely ramified, $3=\mathfrak p_1\mathfrak p_2^2$
with $\mathfrak p_1\neq\mathfrak p_2$ primes of $\OK$, for $v_3(b)$ even. The
total ramification when $3\nmid b$ is what forces the hypothesis $9\mid b$ in the
construction of \S\ref{sec:tower}.

Our first main theorem determines the fundamental unit.

\begin{theoremA}[Fundamental unit; see Theorems~\ref{thm:main}, \ref{thm:inf} and
Corollary~\ref{cor:noniso}]
\label{thmA:fund}
Let $b\in\intZ$ with $3\nmid b$ and $b^3-4$ squarefree, and either $b\ge7$ or
$b\le-1$. Then $\eta_{b}=-\epsilon_b^{-1}=\bigl((b+1)-\theta\bigr)^{-1}$ is the
fundamental unit of $K_b$.
The integers $b$ satisfying these hypotheses have density
$\lambda=0.31250\ldots$ among the positive integers, and the same density among
the negative integers, where $\lambda$ is the Euler product of
Theorem~\ref{thm:inf}; in particular there are infinitely many such $b$, and
distinct such $b$ give non-isomorphic fields $K_b$.
\end{theoremA}

The density comes from Hooley's asymptotic formula for squarefree values of
irreducible cubic polynomials~\cite[Ch.~4, Theorem~3]{Hooley1976}. Erd\H{o}s~\cite{MR56635}
had shown earlier that such values occur infinitely often, but his method gives no
density \cite[p.~62]{Hooley1976}. The proof of fundamentality combines the index
bound just described with a sharp form of an inequality of Artin
(Lemma~\ref{lem:artin}). One may moreover take $b$ to be \emph{prime}, of either
sign (Corollary~\ref{cor:primeb}), by a theorem of Helfgott~\cite{Helfgott2014}
on squarefree values of cubic polynomials at prime arguments.
Since an isomorphism between two fields of the family must carry fundamental unit
to fundamental unit, the parameter $b$ is recovered from $K_b$ whenever $\eta_b$
is fundamental (Corollary~\ref{cor:noniso}).

Our second main theorem describes the exceptions.

\begin{theoremA}[The exceptional set; see Theorems~\ref{thm:iff},
\ref{thm:closed},~\ref{thm:starq} and~\ref{thm:exc}]\label{thmA:exc}
Let $(E_n)_{n\in\intZ}$ be the sequence defined by
\[
  E_n=\tfrac{1}{2\sqrt3}\bigl((2+\sqrt3)^n-(2-\sqrt3)^n\bigr),
\]
and put
\[
  b_{\pm}(n)=\tfrac12\bigl(E_{2n-1}-1\bigr)\pm2E_{n-1}\qquad(n\in\intZ).
\]
Then $\eta_b$ is a square in $\OK^\times$ if and only if
\[
  b\in\set{b_+(n) | n\in \intZ}\cup\set{b_{-}(n) | n\in\intZ},
\]
the degenerate values $b_-(0)=1$ and $b_+(1)=b_-(1)=b_-(-1)=0$ being excluded by
our standing convention $b\neq0,1$; there are infinitely many such $b$. If in
addition
\begin{equation}\label{eq:starA}
  |\disc{K_b}|\ \ge\ 4\bigl(3(b^2+b+1)+1\bigr)^{3/5}+24 ,
\end{equation}
then $\eta_b$ fails to be the fundamental unit \emph{precisely} when it is such a
square. Finally, every such $b$ other than $b=-3$ has $b^3-4$
\emph{non}-squarefree, with a prime $p\nmid b$ satisfying $p^2\mid b^3-4$; such
$p$ are exactly the primes $p\nmid b$ dividing $[\OK:\intZ[\theta]]$, and they are
characterized by $b^3\equiv4\pmod{p^2}$.
\end{theoremA}

The hypothesis~\eqref{eq:starA} is mild: its right-hand side is
$\asymp|b|^{6/5}$, whereas $|\disc{K_b}|$ has size $|b|^4$ for generic $b$, and it
holds throughout the regime of Theorem~\ref{thmA:fund}. Among the $5999$ integers
$b\neq0,1$ with $|b|\le3000$ it fails for four, namely $b=-75,-3,3,9$. Three of
them are square-exceptions anyway; at $b=3$ the equivalence itself fails, so
\eqref{eq:starA} cannot be dropped (Remark~\ref{rem:star}).

The last assertion of Theorem~\ref{thmA:exc} makes precise our observation about
the index of the exceptional $b$: the hypotheses of Theorem~\ref{thmA:fund} exclude
every square-exception automatically, all of them by the squarefreeness of $b^3-4$
except $b=-3$, which is excluded by $3\nmid b$.

Finally, we carry out the analogue of Kaneko's construction: whenever $v_3(b)$ is
even and $\ge2$ \textup(so that $9\mid b$\textup) and $b\neq9$, the biquadratic
field $F_b=\ratQ\bigl(\sqrt{-3},\sqrt{b(b^3-4)}\bigr)$ has $3$-class field tower of
length greater than $1$, and there are infinitely many such $b$ of either sign
(Theorem~\ref{thm:tower}). As in \cite[\S3]{MR3363170} the proof rests on a
criterion of Yoshida~\cite{MR1970518,MR2037577}, applied to the unit $\eta_b$.

\subsection{Organization}
Section~\ref{sec:setup} fixes notation and the basic properties of $K_b$,
$\epsilon_b$ and $\eta_b$; \S\ref{sec:basis} computes the index and the integral
basis; \S\ref{sec:fund} proves Theorem~\ref{thmA:fund}, including the
non-isomorphism of the fields it covers;
\S\ref{sec:exc} proves Theorem~\ref{thmA:exc}; and \S\ref{sec:tower} constructs
the biquadratic fields. These last four sections are, for $f_b(x)$, the analogues
of~\cite{MR2037576}, \cite{MR3363170}, \cite{MR2773123} and~\cite[\S3]{MR3363170}.

The numerical claims below were verified with PARI/GP \cite{PARI2}; the scripts
that verify them are available in a public repository \cite{KimuraCode}, which
records for each script the statements it checks.

\section{Setup and basic properties}\label{sec:setup}

We treat positive and negative $b$ uniformly, writing $c=-b$ for $b\le-1$ when
convenient. Recall that $f_b(x)=x^3-3bx-b^3$ and $\poldisc{f_b(x)}=-27b^3(b^3-4)$.
For an algebraic number $\alpha$ we write $\poldisc{\alpha}$ for the discriminant
of the minimal polynomial of $\alpha$ over $\ratQ$, so that
$\poldisc{\theta}=\poldisc{f_b(x)}$; $\disc{K}$ denotes the discriminant of a number
field $K$; and $\Norm$ abbreviates $\Norm_{K_b/\ratQ}$.

\begin{proposition}\label{prop:setup}
Let $b\in\intZ$ with $b\neq0,1$, let $\theta$ be the unique real root of $f_b(x)$, and
set $K_b=\ratQ(\theta)$, $\epsilon_b=\theta-(b+1)$ and
$\eta_b=-\epsilon_b^{-1}=(b+1-\theta)^{-1}$.
\begin{enumerate}
\item[\textup{(i)}] $K_b$ is a cubic field which is not totally real.
\item[\textup{(ii)}] $\Norm(\epsilon_b)=-1$, hence
$\epsilon_b=\theta-(b+1)$ is a unit of $\OK$ of norm $-1$, for every $b$.
\item[\textup{(iii)}] The real root satisfies $b<\theta<b+1$, so
$\epsilon_b\in(-1,0)$, and $\eta_b>1$.
\end{enumerate}
\end{proposition}
\begin{proof}
Since $b\neq0,1$ we have $b^3(b^3-4)>0$, so $\poldisc{f_b(x)}<0$ and $f_b(x)$ has
exactly one real root $\theta$ and one conjugate pair of non-real roots. We prove
(iii) first and use it for (i).

(iii) $f_b(b)=-3b^2<0<1=f_b(b+1)$, so $f_b(x)$ has a real root in the open interval
$(b,b+1)$; as $\theta$ is its only real root, $b<\theta<b+1$. Hence
$\epsilon_b=\theta-(b+1)\in(-1,0)$ and $\eta_b=-\epsilon_b^{-1}>1$.

(i) By (iii) the only real root $\theta$ lies strictly between the consecutive
integers $b$ and $b+1$, so it is not a rational integer; as $f_b(x)$ is monic in
$\intZ[x]$ it therefore has no rational root, and being cubic it is irreducible.
Thus $K_b=\ratQ(\theta)$ is a cubic field with $r_1=r_2=1$, that is, one which is
not totally real.

(ii) $f_b(b+1)=(b+1)^3-3b(b+1)-b^3=1$. Writing $f_b(x)=\prod_i(x-\theta_i)$,
\[
  \Norm(\epsilon_b)=\prod_i(\theta_i-(b+1))=(-1)^3\prod_i\bigl((b+1)-\theta_i\bigr)
  =-f_b(b+1)=-1 ,
\]
so $\epsilon_b=\theta-(b+1)$ is a unit of norm $-1$.
\end{proof}

Translating by $x\mapsto x+(b+1)$ does not change the discriminant, so
$[\OK:\intZ[\epsilon_b]]=[\OK:\intZ[\theta]]$, and the minimal polynomial of
$\epsilon_b$ is $g(x)=x^3+3(b+1)x^2+3(b^2+b+1)x+1$.

\subsection{The fundamental unit}

By Dirichlet's unit theorem, $\OK^\times$ has rank $1$, and as $K_b$ has a unique
real embedding there is a unique unit $\eta_0$ that is $>1$ in that embedding with
$\OK^\times=\{\pm\eta_0^{\,k}:k\in\intZ\}$. We call $\eta_0$ the \emph{fundamental
unit} of $K_b$. (Every unit $u>1$ of a cubic field $K$ that is not totally real
has $\Norm_{K/\ratQ}(u)=u\,|\sigma_2(u)|^2>0$, where $\sigma_2$ is either of the
two complex embeddings of $K$, hence $\Norm_{K/\ratQ}(u)=+1$; in particular
$\Norm(\eta_0)=+1$.)

By Proposition~\ref{prop:setup}(iii) the unit $\epsilon_b$ lies in $(-1,0)$, so it
is neither positive nor $>1$ and cannot itself be the fundamental unit. Its
negative inverse is the natural candidate:
\[
  \eta_{b}\coloneqq-\frac1{\epsilon_b}=\bigl((b+1)-\theta\bigr)^{-1}>1,
  \quad \eta_b\approx3(b^2+b+1),
\]
a unit of norm $+1$ with minimal polynomial (the reciprocal of $g(x)$,
$x\mapsto-x$)
\[
  h(X)=X^3-PX^2+QX-1,\quad P=3(b^2+b+1),\quad Q=3(b+1).
\]
Since $\epsilon_b=-\eta_b^{-1}$ we have $\{\pm\epsilon_b^{\,k}\}=\{\pm\eta_b^{\,k}\}$,
so $\eta_b$ generates $\OK^\times$ modulo $\pm1$ if and only if $\epsilon_b$ does;
and since $\eta_b>1$, this happens exactly when $\eta_b=\eta_0$. Consequently
\begin{equation}\label{eq:fundiff}
  \eta_b=\eta_0\quad\text{if and only if}\quad
  \eta_b\ \text{is not a proper power}\ u^q\ (q\ \text{prime},\ u>1).
\end{equation}

The following bound on $\eta_b$, used in \S\ref{sec:fund}, is uniform in the sign
of $b$, which the sharper bound $\eta_b<P$ is not (Remark~\ref{rem:etaP}).

\begin{lemma}\label{lem:etasize}
For every $b\neq0,1$ one has $1<\eta_b<P+1$, where $P=3(b^2+b+1)$.
\end{lemma}
\begin{proof}
$\eta_b$ is the unique real root of $h(X)=X^3-PX^2+QX-1$. Now
$h(1)=Q-P=-3b^2<0$, and
\[
  h(P+1)=(P+1)(P+1+Q)-1>0,
\]
because $P+1+Q=3b^2+6b+7=3(b+1)^2+4>0$. Hence $1<\eta_b<P+1$.
\end{proof}

\begin{remark}\label{rem:etaP}
For $b\ge2$ one has the sharper $\eta_b<P$, since then $h(P)=QP-1>0$; but for
$b\le-1$ we have $Q=3(b+1)\le0$, so $h(P)\le-1<0$ and in fact $\eta_b>P$
\textup(e.g.\ $\eta_{-2}=9.3329\ldots>9=P$, and $\eta_{-100}=29703.0099\ldots>29703=P$\textup).
This is why Lemma~\ref{lem:etasize} is stated with $P+1$.
\end{remark}

\section{The integral basis and the index \texorpdfstring{$[\OK:\intZ[\theta]]$}{}}
\label{sec:basis}

By Proposition~\ref{prop:setup} and the translation invariance of the
discriminant, determining $\OK$ amounts to computing the single index
$a\coloneqq[\OK:\intZ[\theta]]$, which we do prime by prime; an integral basis is
then exhibited in Corollary~\ref{cor:idxmain}.

\begin{remark}\label{rem:voronoi}
Voronoi's classical theorem (\cite[\S17]{MR160744}; see also \cite{MR2037576}) is
not used below: the index is computed prime by prime in
Theorems~\ref{thm:idxprime} and~\ref{thm:idx23} from Newton polygons and
Dedekind's criterion, and the integral basis of Corollary~\ref{cor:idxmain} is
verified directly by Lemma~\ref{lem:omega}.
\end{remark}

\begin{remark}\label{rem:reduced}
  A cubic algebraic integer $\delta$ with minimal polynomial $x^3-mx-n$,
  $m,n\in\intZ$, is called \emph{reduced} if no prime $p$ satisfies $p^2\mid m$
  and $p^3\mid n$.
  In our family $f_b(x)$ is reduced if and only if $3\nmid b$ and $b$ is squarefree,
  since $p^2\mid3b$ and $p^3\mid b^3$ hold simultaneously exactly when $p^2\mid b$
  for $p\neq3$, or when $p=3$ and $3\mid b$; this is the squarefreeness hypothesis
  of Corollary~\ref{cor:idxmain}. Being reduced is not a formality: for $b=25$ one
  has $a=125$, and $\theta/5$ is already an algebraic integer \textup(indeed
  $\{1,\theta/5,\theta^2/25\}$ is an integral basis of $\OK$\textup), so \emph{no}
  integral basis of $\OK$ contains $\theta$.
\end{remark}

We now compute $v_p(a)$ prime by prime from the relation
\[
  v_p(\poldisc{f_b(x)})=2v_p(a)+v_p(\disc{K_b}),
\]
between the two discriminants.

\begin{theorem}\label{thm:idxprime}
For a prime $p\ge5$,
\[
  v_p(a)=\begin{cases}
   \big\lfloor \tfrac32 v_p(b)\big\rfloor, & p\mid b,\\[2pt]
   \big\lfloor \tfrac12 v_p(b^3-4)\big\rfloor, & p\nmid b.
  \end{cases}
\]
\end{theorem}
\begin{proof}
As $\gcd(b,b^3-4)\mid4$, for $p\ge5$ at most one of $p\mid b$, $p\mid b^3-4$ holds.

\emph{Case $p\mid b$, $v_p(b)=w$.} The $p$-adic Newton polygon of
$f_b(x)=x^3-3bx-b^3$ has vertices $(0,3w)$, $(1,w)$, $(3,0)$ (since $v_p(3b)=w$,
$v_p(b^3)=3w$): a segment of slope $-2w$ and horizontal length $1$ (one root of
valuation $2w$) and a segment of slope $-w/2$ and length $2$ (two roots of
valuation $w/2$). The first yields an unramified linear factor over the field
$\ratQp$ of $p$-adic numbers; the
second yields an unramified quadratic factor if $w$ is even and one with
ramification index $e=2$, tame as $p\ge5$, if $w$ is odd. Hence $v_p(\disc{K_b})=w\bmod2$, and from
$v_p(\poldisc{f_b(x)})=3w$ we get $2v_p(a)=3w-(w\bmod2)$, i.e.\ $v_p(a)=\lfloor3w/2\rfloor$.

\emph{Case $p\nmid b$, $v\coloneqq v_p(b^3-4)>0$.} Then $f_b(x)\bmod p$ has a double root.
From $f_b'(x)=3(x^2-b)$, a repeated root has $x^2\equiv b$ and
$0\equiv f_b(x)\equiv-2bx-b^3$, giving the double root $x_0\equiv-b^2/2$ and a simple
root $x_1\equiv b^2$ (distinct, as $x_1-x_0\equiv\tfrac32 b^2\not\equiv0$). The
simple root splits off an unramified linear factor. At the double root,
$\intZ_p[\theta]$, with $\intZ_p$ the ring of $p$-adic integers, sits at index
$p^{\lfloor v/2\rfloor}$ below the maximal local
order, with the place unramified if $v$ is even and ramified ($e=2$, tame) if $v$
is odd: indeed iterating Dedekind's criterion (equivalently, the Newton polygon of
$f_b(x)$ recentered at a lift of $x_0$) removes a factor $p^2$ from the discriminant
at each step until $v<2$. Hence $v_p(\disc{K_b})=v\bmod2$ and
$v_p(a)=\lfloor v/2\rfloor$.
\end{proof}

The primes $2$ and $3$ require separate treatment; the following standard lemma
handles the wild prime $2$.

\begin{lemma}\label{lem:q2}
  
Let $\delta\in\ratQ_2^\times$ be a non-square, $L=\ratQ_2(\sqrt\delta)$, and
$u\coloneqq\delta/2^{v_2(\delta)}\in\intZ_2^\times$, and let $d_{L/\ratQ_2}$,
abbreviated $d_L$, be the discriminant of $L/\ratQ_2$. Then
$v_2(d_{L/\ratQ_2})=3$ if $v_2(\delta)$ is odd. If $v_2(\delta)$ is even, then
$v_2(d_{L/\ratQ_2})=2$ when
$u\equiv3,7\pmod8$, and $L/\ratQ_2$ is unramified \textup($v_2(d_L)=0$\textup) when
$u\equiv5\pmod8$ \textup($u\equiv1\pmod8$ forces $\delta$ a square\textup).
\end{lemma}
\begin{proof}
This is an easy consequence of the classification of quadratic extensions of
$\ratQ_2$ (see e.g.\ \cite[Ch.~II, \S3.3]{MR49:8956}).
\end{proof}

\begin{theorem}\label{thm:idx23}
For $p=3$: $\ v_3(a)=0$ if $3\nmid b$, and $v_3(a)=\big\lfloor\tfrac32(v_3(b)+1)\big\rfloor$
if $3\mid b$. For $p=2$: $\ v_2(a)=0$ if $b$ is odd, and for $b$ even with
$w=v_2(b)$ and $b'=b/2^w$,
\[
  v_2(a)=\begin{cases}(3w-1)/2,& w\ \text{odd},\\[2pt]
  3w/2+[\,b'\equiv3\!\!\pmod4\,],& w\ \text{even}.\end{cases}
\]
\end{theorem}
\begin{proof}
Throughout $v_p(\poldisc{f_b(x)})=2v_p(a)+v_p(\disc{K_b})$, so it suffices to find
$v_p(\disc{K_b})$.

\emph{$p=3$, $3\nmid b$.} Then $f_b(x)\equiv x^3-b^3\equiv(x-b)^3\pmod3$. In
Dedekind's criterion take $\bar g(x)=x-b$, $\bar h(x)=(x-b)^2$; then
$M(x)=(f_b(x)-(x-b)^3)/3=b\,x\,(x-(1+b))$ and $\bar M(b)=-b^2\not\equiv0\pmod3$,
so $\gcd\bigl(\bar g(x),\bar h(x),\bar M(x)\bigr)=1$ and $3\nmid a$. Then
$v_3(\disc{K_b})=v_3(\poldisc{f_b(x)})=3+v_3(b^3-4)\ge3$. Were $3$ not totally
ramified, $f_b(x)$ would factor over $\ratQ_3$ as a linear times a quadratic factor,
the latter unramified or tamely ramified as $3\nmid2$, giving
$v_3(\disc{K_b})\le1$. So for $3\nmid b$ the prime $3$ is totally ramified in
$K_b$, and wildly, $e=p=3$.

\emph{$p=3$, $3\mid b$, $w=v_3(b)$.} The element $y=\theta/3$ is an algebraic
integer: from $\theta^3=3b\theta+b^3$ one gets $y^3=By+B^3$ with $B=b/3$, so $y$
is a root of $h_B(Y)=Y^3-BY-B^3$ and $\ratQ(y)=K_b$. Since
$\intZ[\theta]=\intZ[3y]\subseteq\intZ[y]$ with $[\intZ[y]:\intZ[\theta]]=27$ (a Vandermonde
computation), $v_3(a)=v_3([\OK:\intZ[y]])+3$. Now $\poldisc{h_B(Y)}=-B^3(27B^3-4)$
has
$v_3=3w'$ with $w'=v_3(B)=w-1$. If $w'=0$ then $\intZ[y]$ is $3$-maximal and
$v_3([\OK:\intZ[y]])=0$. If $w'\ge1$, the Newton polygon of $h_B(Y)$ at $3$ has
vertices $(0,3w'),(1,w'),(3,0)$, giving a $\ratQ_3$-rational linear factor (root
$y_1$, $v_3=2w'$) and a quadratic factor $\kappa(Y)=Y^2+y_1Y+c$, $c=B^3/y_1$,
$v_3(c)=w'$; so $v_3(\poldisc{\kappa(Y)})=v_3(y_1^2-4c)=\min(4w',w')=w'$. As
$3\nmid2$, $\kappa(Y)$ defines a tamely ramified ($e=2$) quadratic when $w'$ is odd and an
unramified/split one when $w'$ is even, so $v_3(\disc{K_b})=[w'\ \text{odd}]$ and
$v_3([\OK:\intZ[y]])=(3w'-[w'\ \text{odd}])/2=\lfloor3w'/2\rfloor$. Hence
$v_3(a)=\lfloor3(w-1)/2\rfloor+3=\lfloor3(w+1)/2\rfloor$ (also valid when $w=1$).

\emph{$p=2$, $b$ odd.} Then $b^3-4$ is odd and $v_2(\poldisc{f_b(x)})=0$, so $2\nmid a$.

\emph{$p=2$, $b$ even, $w=v_2(b)\ge1$, $b'=b/2^w$ odd.} The Newton polygon of
$f_b(x)$ at $2$ has vertices $(0,3w),(1,w),(3,0)$: a $\ratQ_2$-rational linear factor
with root $\theta_1$, $v_2(\theta_1)=2w$, and a quadratic factor
$\kappa(x)=x^2+\theta_1x+c$ with $c=b^3/\theta_1$, $v_2(c)=w$ (from
$\theta_2+\theta_3=-\theta_1$, $\theta_2\theta_3=c$). Then
$v_2(\poldisc{\kappa(x)})=v_2(\theta_1^2-4c)=\min(4w,w+2)=w+2$. Writing $\theta_1=2^{2w}s$,
$f_b(\theta_1)=0$ gives $s\equiv-b'^2/3\pmod{2^{3w}}$.
If $w$ is odd, $v_2(\poldisc{\kappa(x)})=w+2$ is odd, so by Lemma~\ref{lem:q2}
$v_2(\disc{K_b})=3$ and $v_2(a)=(3w+2-3)/2=(3w-1)/2$.
If $w$ is even, the unit $u=\poldisc{\kappa(x)}/2^{w+2}\equiv-c/2^w=-b'^3/s\equiv3b'\pmod8$
(the $\theta_1^2$ term contributes $\equiv0\bmod8$ after dividing by $2^{w+2}$,
as $3w-2\ge4$). By Lemma~\ref{lem:q2}, $b'\equiv1\pmod4$ (so $3b'\equiv3,7\pmod8$)
gives $v_2(\disc{K_b})=2$, $v_2(a)=3w/2$; and $b'\equiv3\pmod4$ (so $3b'\equiv1,5\pmod8$)
gives $v_2(\disc{K_b})=0$, $v_2(a)=3w/2+1$. Both equal $3w/2+[b'\equiv3\!\!\pmod4]$.
\end{proof}

\begin{remark}
In the notation of Llorente-Nart~\cite{MR687621} ($f=X^3-AX+B$,
$\Delta=4A^3-27B^2=i(\theta)^2D$), our family is $A=3b$, $B=-b^3$. Their reduction
``$\theta\mapsto\theta/p$ when $v_p(A)\ge2,v_p(B)\ge3$'' is precisely the
substitution $y=\theta/3$ used above, and their Theorem~2 at $p=2$ reduces, with
$v_2(\Delta)=3w+2$, to exactly the dichotomy of Theorem~\ref{thm:idx23}.
\end{remark}

The integral basis is now obtained from the following computation, which requires
no hypothesis on $b$.

\begin{lemma}\label{lem:omega}
For every $b\neq0,1$ the element
\[
  \omega\coloneqq\frac{\theta^2-3b}{b}\in K_b
\]
is an algebraic integer, and its minimal polynomial over $\ratQ$ is
\[
  X^3+3X^2-b^3\in\intZ[X].
\]
\end{lemma}
\begin{proof}
One can check directly that $\omega$ is a root of $X^3+3X^2-b^3$.
One can see, from $\theta^2=b(\omega+3)$, that
$\ratQ(\theta^2)\subseteq\ratQ(\omega)\subseteq K_b$, and $\theta^2\notin\ratQ$
because $\theta$ has degree $3$; as $[K_b:\ratQ]=3$ is prime this forces
$\ratQ(\omega)=K_b$, so $\omega$ has degree $3$ and $X^3+3X^2-b^3$ is its minimal
polynomial.
\end{proof}

\begin{remark}\label{rem:sister}
Substituting $X=Z-1$ in $X^3+3X^2-b^3$ gives $Z^3-3Z-(b^3-2)$. Hence
$\zeta\coloneqq\omega+1=(\theta^2-2b)/b$ satisfies
\[
  \zeta^3-3\zeta=b^3-2,
\]
so that $K_b=\ratQ(\zeta)$ is the member $n=b^3-2$ of the family $x^3-3x-n$,
whereas Ishida's family~\eqref{eq:sister} is its member $n=b^3$. The discriminant
of $x^3-3x-n$ is $-27(n-2)(n+2)$, specializing to
$-27b^3(b^3-4)$ for $n=b^3-2$ and to $-27(b^3-2)(b^3+2)$ for $n=b^3$. In
particular $x^3-3x-(b^3-2)$ is irreducible for every $b\neq0,1$.
\end{remark}

\begin{corollary}\label{cor:idxmain}
If $3\nmid b$ and $b^3-4$ is squarefree \textup(which forces $b$ odd\textup), then
\[
  a=[\OK:\intZ[\theta]]=\prod_{p\mid b}p^{\lfloor 3v_p(b)/2\rfloor}.
\]
If moreover $b$ is squarefree, then $a=|b|$ and
\[
  \Bigl\{\,1,\ \theta,\ \frac{\theta^2-3b}{b}\,\Bigr\}
\]
is an integral basis of $\OK$, with $\disc{K_b}=-27b(b^3-4)$.
\end{corollary}
\begin{proof}
Squarefreeness of $b^3-4$ forces $b$ odd (else $4\mid b^3-4$) and kills the
$p\nmid b$ terms of Theorem~\ref{thm:idxprime} ($\lfloor v_p(b^3-4)/2\rfloor=0$);
$3\nmid b$ gives $v_3(a)=0$ and $b$ odd gives $v_2(a)=0$ by
Theorem~\ref{thm:idx23}. Theorem~\ref{thm:idxprime} leaves
$a=\prod_{p\mid b}p^{\lfloor3v_p(b)/2\rfloor}$, which equals $|b|$ when $b$ is
squarefree.

Assume now that $b$ is squarefree, and let $\omega=(\theta^2-3b)/b$ be as in
Lemma~\ref{lem:omega}, so that $\Lambda\coloneqq\intZ+\intZ\theta+\intZ\omega$ is a
sublattice of $\OK$ of rank three. The transition matrix from
$(1,\theta,\theta^2)$ to $(1,\theta,\omega)$ is lower triangular with diagonal
entries $1,1,1/b$, so its determinant is $1/b$ and
\[
  \disc{\Lambda}=\frac{\poldisc{f_b(x)}}{b^2}=\frac{-27b^3(b^3-4)}{b^2}=-27b(b^3-4).
\]
On the other hand $a=|b|$ gives
$\disc{K_b}=\poldisc{f_b(x)}/a^2=-27b(b^3-4)$ as well. Since $\Lambda\subseteq\OK$ and
$\disc{\Lambda}=[\OK:\Lambda]^2\disc{K_b}$ with $\disc{K_b}\neq0$, we get
$[\OK:\Lambda]=1$, i.e.\ $\Lambda=\OK$ and $\{1,\theta,\omega\}$ is an integral basis.
\end{proof}

\begin{remark}
Since Lemma~\ref{lem:omega} is unconditional, so is the computation
$\disc{\Lambda}=\poldisc{f_b(x)}/b^2$; comparing it with $\disc{K_b}=\poldisc{f_b(x)}/a^2$
shows that
\[
  [\OK:\Lambda]=\frac{a}{|b|}\qquad\text{for every }b\neq0,1 .
\]
Thus $\Lambda$ is maximal precisely when $a=|b|$; for the non-reduced $b=25$ of
Remark~\ref{rem:reduced} one gets $[\OK:\Lambda]=5$.
\end{remark}

\begin{example}
For $b=15=3\cdot5$ one gets $a=135=3^3\cdot5$ and the integral basis
$\{1,\theta/3,(\theta^2-45)/45\}$; for the even $b=8=2^3$, $v_2(b)=3$ is odd, so
$2$ is wildly ramified, $a=2^4=16$, and $\{1,\theta/2,(\theta^2-16)/8\}$ is an
integral basis. Both are outside the regime of Corollary~\ref{cor:idxmain}.
\end{example}

We record a consequence used in the next section. In the regime of
Corollary~\ref{cor:idxmain} one has $a^2\mid|b|^3$, so
\begin{equation}\label{eq:dkbound}
  |\disc{K_b}|=\frac{27\,|b|^3\,|b^3-4|}{a^2}\ \ge\ 27\,|b^3-4| ,
\end{equation}
that is, $|\disc{K_b}|\ge27(b^3-4)$ for $b\ge2$, and $|\disc{K_b}|\ge27(c^3+4)$ for
$b=-c\le-1$.

\section{The fundamental unit}\label{sec:fund}

By~\eqref{eq:fundiff}, $\eta_b$ is the fundamental unit if and only if it is not a
proper power of a unit $>1$. In the regime of Corollary~\ref{cor:idxmain} the
discriminant of $K_b$ is large enough to exclude every prime exponent at once;
the two inequalities behind this are proved first.

\subsection{Two inequalities}

The first is an inequality of Artin, for which we include a proof.

\begin{lemma}[Artin's inequality; cf.~{\cite[Lemma~2]{MR335469}}]\label{lem:artin}
Let $K$ be a non-totally real cubic field and $e\in\intring{K}^\times$ with $e>1$ in the
real embedding. Then $|\poldisc{e}|\le4e^3+24$; in particular
$|\disc{K}|\le4e^3+24$.
\end{lemma}
\begin{proof}
As $e>1$ is irrational it generates $K$, so $\poldisc{e}=[\intring{K}:\intZ[e]]^2\disc{K}$
and $|\disc{K}|\le|\poldisc{e}|$. Let $\sigma_1$ be the real embedding, so that $\sigma_1(e)=e$,
and let $\sigma_2,\sigma_3$ be the two complex embeddings, say
$\sigma_2(e)=\rho\exp(\sqrt{-1}\phi)$ and $\sigma_3(e)=\rho\exp(-\sqrt{-1}\phi)$
with $\rho>0$ and $\phi\in(0,\pi)$. The norm $e\rho^2=\pm1$ is positive, hence
equals $1$, so $\rho=e^{-1/2}$. Then
\[
  \poldisc{e}=\prod_{i<j}(\sigma_i e-\sigma_j e)^2
      =-4e^{-1}\sin^2\!\phi\,(e^2+e^{-1}-2e^{1/2}\cos\phi)^2<0.
\]
In terms of $s=e^{3/2}>1$ and $t=\cos\phi\in[-1,1]$ this reads
$|\poldisc{e}|=4s^{-2}(1-t^2)(s^2-2st+1)^2$, so the assertion
$|\poldisc{e}|\le4s^2+24$ is, after multiplication by $s^2$, that
\[
  F(s,t)\coloneqq4s^4+24s^2-4(1-t^2)(s^2-2st+1)^2\ge0.
\]
One verifies the identity
\[
  F(s,t)=G(s,t)^2+16t^2s^2-4,\quad G(s,t)=2ts^2+4(1-t^2)s+2t.
\]
If $16t^2s^2\ge4$, then $F\ge0$ at once. Otherwise set $\gamma=2|t|s\in[0,1)$; we
must show $G^2\ge4-4\gamma^2$. With $|t|=\gamma/(2s)$,
\[
  G\ge4(1-t^2)s-2|t|(s^2+1)=(4-\gamma)s-\tfrac{\gamma^2+\gamma}{s}=:H(s),
\]
and $H'(s)>0$, so $H(s)\ge H(1)=4-2\gamma-\gamma^2\ge1>0$. Hence
\[
  G^2-(4-4\gamma^2)\ge(4-2\gamma-\gamma^2)^2-4(1-\gamma^2)
   =\gamma^4+4\gamma^3-16\gamma+12=:\psi(\gamma),
\]
and $\psi'(\gamma)=4(\gamma-1)(\gamma+2)^2\le0$ on $[0,1]$ gives
$\psi(\gamma)\ge\psi(1)=1>0$. Thus $F\ge0$, proving the claim.
\end{proof}

\begin{remark}
The bound is sharp: as $e\to\infty$ the optimal angle has $\cos\phi\sim-2e^{-3/2}$
and $\max_\phi|\poldisc{e}|=4e^3+24-o(1)$, so $24$ cannot be lowered.
\end{remark}

The second inequality is elementary.

\begin{lemma}\label{lem:size}
\textup{(i)} For every real $b\ge7$,
$\ 27(b^3-4)>4\bigl(3(b^2+b+1)+1\bigr)^{3/2}+24$.
\textup{(ii)} For every real $c\ge2$, and also for $c=1$,
$\ 27(c^3+4)>4\bigl(3(c^2-c+1)+1\bigr)^{3/2}+24$.
\end{lemma}
\begin{proof}
(i) Both sides are positive and $27(b^3-4)-24=27b^3-132>0$ for $b\ge7$, so squaring,
the claim is $\Phi(b)\coloneqq(27b^3-132)^2-16(3b^2+3b+4)^3>0$, i.e.
\[
  \Phi(b)=297b^6-1296b^5-3024b^4-11016b^3-4032b^2-2304b+16400>0.
\]
Bounding each negative term by $b^k\le b^6/7^{6-k}$, valid for $b\ge7$, leaves
$\Phi(b)\ge\tfrac{272439}{16807}b^6+16400>0$.

(ii) Likewise $27(c^3+4)-24=27c^3+84>0$, so squaring, the claim is
$\Psi(c)\coloneqq(27c^3+84)^2-16(3c^2-3c+4)^3>0$, i.e.
\[
  \Psi(c)=297c^6+1296c^5-3024c^4+8424c^3-4032c^2+2304c+6032>0.
\]
For $c\ge2$ the bounds $1296c^5-3024c^4\ge-432c^4\ge-108c^6$ and
$4032c^2\le2016c^3$ leave $\Psi(c)\ge189c^6+6408c^3+2304c+6032>0$; and
$\Psi(1)=11297>0$, the asserted inequality reading $135>56$ there.
\end{proof}

\subsection{The main theorem}

\begin{theorem}[Fundamental unit]\label{thm:main}
Let $b$ satisfy $3\nmid b$ and $b^3-4$ squarefree, and let either $b\ge7$ or
$b\le-1$. Then $\eta_b$ is the fundamental unit of $K_b$.
\end{theorem}
\begin{proof}
Suppose $\eta_b=u^q$ for a unit $u>1$ and a prime $q\ge2$. Then
$u\le\eta_b^{1/q}\le\eta_b^{1/2}$, so Lemma~\ref{lem:etasize} gives
$u^3<\bigl(3(b^2+b+1)+1\bigr)^{3/2}$, and Lemma~\ref{lem:artin}, applied to $u$,
yields
\begin{equation}\label{eq:mainchain}
  |\disc{K_b}|\le4u^3+24<4\bigl(3(b^2+b+1)+1\bigr)^{3/2}+24 .
\end{equation}
For $b\ge7$, \eqref{eq:dkbound} gives $|\disc{K_b}|\ge27(b^3-4)$, and the
right-hand side of~\eqref{eq:mainchain} is smaller by Lemma~\ref{lem:size}(i).
For $b=-c\le-1$ it gives $|\disc{K_b}|\ge27(c^3+4)$, while $b^2+b+1=c^2-c+1$, and
the right-hand side of~\eqref{eq:mainchain} is smaller by
Lemma~\ref{lem:size}(ii). Either
way~\eqref{eq:mainchain} contradicts the lower bound, so $\eta_b$ is not a proper
power and~\eqref{eq:fundiff} gives the theorem. \textup(In the extreme case
$b=-1$, where $a=1$, $\disc{K_{-1}}=-135$ and
$\eta_{-1}=3.104\ldots$, the chain reads $135\le56$.\textup)
\end{proof}

\begin{lemma}\label{lem:nofixsq}
Let $f(x)=x^3-4$ or $f(x)=x^3+4$. For every prime $q$ there is an
$x\in(\intZ/q^2\intZ)^\times$ with $f(x)\not\equiv0\pmod{q^2}$; in particular $f$ has
no fixed square divisor.
\end{lemma}
\begin{proof}
For $q=2$ take $x=1$: both $1-4$ and $1+4$ are $\equiv1\pmod4$. For $q=3$ the
cubes of the units modulo $9$ are $\pm1$, so $x^3\mp4\equiv\pm1\mp4\not\equiv0
\pmod9$ for every unit $x$. For $q\ge5$ the congruence $x^3\equiv\pm4\pmod{q^2}$
has at most three solutions, by Hensel's lemma applied to $x^3\mp4$ (whose
derivative $3x^2$ is a unit at any solution, as $q\nmid3\cdot4$), while
$\#(\intZ/q^2\intZ)^\times=q(q-1)\ge20$.
\end{proof}

\begin{theorem}\label{thm:inf}
For a prime $q$ let $\nu(q)$ be the number of $x\in\intZ/q\intZ$ with
$x^3\equiv4\pmod q$, and put
\begin{equation*}
  \lambda=\frac13\prod_{q\ge5}\Bigl(1-\frac{\nu(q)}{q^2}\Bigr)=0.31250\ldots,
\end{equation*}
the product being over primes. Then, as $X\to\infty$,
\begin{equation*}
\#\set{1\le b\le X | 3\nmid b,\ b^3-4\ \text{squarefree}}
  =\lambda X+O\bigl(X/\sqrt{\log X}\bigr),
\end{equation*}
and the same holds for $\#\set{-X\le b\le-1 | 3\nmid b,\ b^3-4\ \text{squarefree}}$.
Consequently, for each of the two signs, the number of $b$ with $0<\pm b\le X$
for which $\eta_b$ is the fundamental unit of $K_b$ is at least
$\lambda X+O(X/\sqrt{\log X})$.
\end{theorem}
\begin{proof}
For $b=-c\le-1$ we have $|b^3-4|=c^3+4$, so the second count is the number of
$1\le c\le X$ with $3\nmid c$ and $c^3+4$ squarefree. Fix a sign and
$r\in\{1,2\}$, and let $\Psi(j)=(3j+r)^3\mp4\in\intZ[j]$, with the upper sign for
$b>0$ and the lower for $c$. It is irreducible, since $x^3\mp4$ is and $3j+r$ is
linear. For $\ell\ge1$ let $\nu_\Psi(\ell)$ be the number of $j\in\intZ/\ell\intZ$
with $\Psi(j)\equiv0\pmod\ell$. Hooley's theorem~\cite[Ch.~4, Theorem~3]{Hooley1976},
for an irreducible polynomial of degree~$3$, gives
\begin{equation*}
\#\set{1\le j\le Y | \Psi(j)\ \text{squarefree}}
  =Y\prod_q\Bigl(1-\frac{\nu_\Psi(q^2)}{q^2}\Bigr)+O\bigl(Y/\sqrt{\log Y}\bigr).
\end{equation*}
The local factors are as follows. At $q=2$: if $3j+r$ is odd, so is $\Psi(j)$; if
it is even, $8\mid(3j+r)^3$ and $\Psi(j)\equiv4\pmod8$. Hence $\nu_\Psi(4)=2$. At
$q=3$: $\Psi(j)\equiv r^3\mp4\pmod9$, and $r^3\mp4\in\{-3,5,4,12\}$ is never
$\equiv0\pmod9$, so $\nu_\Psi(9)=0$. At $q\ge5$: $j\mapsto3j+r$ is a bijection of
$\intZ/q^2\intZ$, so $\nu_\Psi(q^2)$ is the number of $x\in\intZ/q^2\intZ$ with
$x^3\equiv\pm4\pmod{q^2}$. This equals $\nu(q)$, by Hensel's lemma (the derivative
$3x^2$ is a unit at every root) and because $x\mapsto-x$ exchanges the two signs.
The product is therefore $\frac12\prod_{q\ge5}(1-\nu(q)/q^2)=\frac32\lambda$, which
is positive since $\nu(q)\le3<q^2$ and $\sum_q\nu(q)/q^2<\infty$. The integers
$1\le b\le X$ with $b\equiv r\pmod3$ are the $b=3j+r$ with $0\le j\le(X-r)/3$.
Taking $Y=(X-r)/3$ and adding the cases $r=1,2$ gives $\lambda X+O(X/\sqrt{\log X})$
for both counts. Of the $b$ counted, only $b=1,2,4,5$ violate $b\ge7$ or $b\le-1$,
so Theorem~\ref{thm:main} applies to all the others.
\end{proof}

A set of positive density need not contain any prime, so the following refinement
is not a consequence of Theorem~\ref{thm:inf}. It rests on a theorem of
Helfgott~\cite{Helfgott2014}: if $f(x)\in\intZ[x]$ is a cubic polynomial with no
repeated roots such that for every prime $q$ the congruence $f(x)\equiv0
\pmod{q^2}$ fails for at least one $x\in(\intZ/q^2\intZ)^\times$, then $f(p)$ is
squarefree for infinitely many primes $p$.

\begin{corollary}\label{cor:primeb}
There are infinitely many primes $p$ such that, for $b=p$, the unit $\eta_b$ is
the fundamental unit of $K_b$ and
$\bigl\{1,\theta,(\theta^2-3b)/b\bigr\}$ is an integral basis of $\OK$, with
$\disc{K_b}=-27b(b^3-4)$. The same holds for $b=-p$ for infinitely many primes $p$.
\end{corollary}
\begin{proof}
Neither $x^3-4$ nor $x^3+4$ has a repeated root, their discriminants being
$-27\cdot16$, so Lemma~\ref{lem:nofixsq} supplies the hypotheses of Helfgott's
theorem. Hence $p^3-4$ is squarefree for infinitely many primes $p$, and so is
$p^3+4=|(-p)^3-4|$. Discard the finitely many $p\le5$; the remaining $p$ are
squarefree with $3\nmid p$, so Corollary~\ref{cor:idxmain} applies to $b=p$ and to
$b=-p$, giving $a=p$, the integral basis and the discriminant, and
Theorem~\ref{thm:main} gives fundamentality, $b=p$ being $\ge7$ and $b=-p$ being
$\le-1$.
\end{proof}

\begin{remark}
By Theorem~\ref{thmA:exc} no such $b$ is a square-exception; consistently, a
prime $p\ge5$ never occurs in the list of Theorem~\ref{thm:closed}. Numerically,
of the $2260$ primes $5\le p<20000$, exactly $2102$ have $p^3-4$ squarefree and
$2094$ have $p^3+4$ squarefree, i.e.\ $93.0\%$ and $92.7\%$; the expected density is
\[
  \prod_q\Bigl(1-\frac{\varrho(q)}{q^2-q}\Bigr)=0.926485\ldots ,
\]
the same for both signs, where $\varrho(q)$ denotes the number of
$x\in(\intZ/q^2\intZ)^\times$ with $x^3\equiv\pm4\pmod{q^2}$.
\end{remark}

\begin{corollary}\label{cor:noniso}
Let $b$ and $b'$ both satisfy the hypotheses of Theorem~\ref{thm:main}. If
$K_b\cong K_{b'}$, then $b=b'$. In particular, Theorem~\ref{thm:inf} and
Corollary~\ref{cor:primeb} provide infinitely many pairwise non-isomorphic fields
$K_b$, each with the explicit fundamental unit $\eta_b$.
\end{corollary}
\begin{proof}
Let $\sigma\colon K_{b'}\to K_b$ be an isomorphism, and let $\rho$ and $\rho'$ be
the real embeddings of $K_b$ and $K_{b'}$. Then $\rho\circ\sigma$ is a real
embedding of $K_{b'}$, so $\rho\circ\sigma=\rho'$ by uniqueness; hence a unit $u$
of $K_{b'}$ is $>1$ in the real embedding if and only if $\sigma(u)$ is. Since
$\sigma$ maps $\intring{K_{b'}}^\times$ onto $\OK^\times$, it maps the fundamental
unit of $K_{b'}$ to that of $K_b$. By Theorem~\ref{thm:main} these are
$\eta_{b'}$ and $\eta_b$, so $\sigma(\eta_{b'})=\eta_b$, and the two units have
the same minimal polynomial. Comparing the coefficients of $X$ in
$h(X)=X^3-3(b^2+b+1)X^2+3(b+1)X-1$ gives $3(b'+1)=3(b+1)$, hence $b=b'$.
\end{proof}

\section{The exceptional set and the index criterion}\label{sec:exc}

We now describe the $b$ for which $\eta_b$ is \emph{not} the fundamental unit.
By~\eqref{eq:fundiff} this amounts to describing the $b$ for
which $\eta_b$ is a proper power $u^q$ with $q$ prime and $u>1$ a unit. The case
$q=2$ is governed by a Pell equation (Theorem~\ref{thm:iff}); the case $q=3$
occurs only for $b=9$ (Lemma~\ref{lem:cube9}); and all $q\ge5$ are excluded by an explicit
mild lower bound on $|\disc{K_b}|$ (Theorem~\ref{thm:starq}). Thus, under that
bound, ``not fundamental'' and ``square'' coincide.

\subsection{The square case is a Pell equation}

\begin{theorem}\label{thm:iff}
Let $b\neq0,1$. Then $\eta_b$ is a square in $\OK^\times$ if and only if the system
\[
  \text{\rm(I)}\ B^2-2A=3(b^2+b+1),\quad
  \text{\rm(II)}\ A^2-2B=3(b+1)
\]
has a solution $(A,B)\in\intZ^2$. Moreover, for a given $A\in\intZ$, an integer $B$
with \textup{(I)},\textup{(II)} exists if and only if $3A(A-2)$ is the square of an
integer, necessarily of the form $(3E)^2$ with $E\in\intZ$; equivalently
\[
  D^2-3E^2=1,\qquad D\coloneqq A-1 .
\]
In that case $A=1\pm D$ and, for a sign $\sigma=\pm1$,
\begin{equation}\label{eq:branch}
  B=2A^2-3-3\sigma E(A+1),\qquad
  b=\tfrac13\bigl(A^2-2B\bigr)-1=1-A^2+2\sigma E(A+1).
\end{equation}
\end{theorem}
\begin{proof}
Assume $\eta_b=\alpha^2$ with $\alpha\in\OK^\times$. Replacing $\alpha$ by $-\alpha$
if necessary we may assume $\alpha>1$ in the real embedding, and then
$\Norm(\alpha)=\alpha|\sigma_2(\alpha)|^2>0$, so $\Norm(\alpha)=1$. The minimal
polynomial of $\alpha$ is thus $h_0(X)=X^3-BX^2+AX-1$ with $A,B\in\intZ$, and
Newton's identities for the squares of its conjugates give (I),(II).

Next, assume $(A,B)$ solves (I), (II).
Set $h_0(X)=X^3-BX^2+AX-1$ with roots
$\alpha_i$ ($e_1=B,e_2=A,e_3=1$). The monic cubic with roots $\alpha_i^2$ is
$X^3-(B^2-2A)X^2+(A^2-2B)X-1=h(X)$ by (I),(II). The polynomial $h_0(X)$ is
irreducible: a rational root is $\pm1$, and $h_0(1)=A-B$, $h_0(-1)=-(A+B+2)$;
$A=B$ forces (via (I)$-$(II)$=3b^2$) $b=0$, and $A+B=-2$ likewise forces $b=0$, both excluded. As
$K_b$ is not totally real, $h(X)$ has one real root $\eta_b$; squares of reals are
real, so $h_0(X)$ has exactly one real root $\alpha$ with $\alpha^2=\eta_b$. Then
$\ratQ(\alpha)$ is cubic containing $\eta_b$, hence $\ratQ(\alpha)=K_b$, and $\alpha\in\OK$ is a unit
($\Norm\alpha=e_3=1$). Thus $\eta_b=\alpha^2$ is a square.

It remains to eliminate $b$ from
(I), (II). Since $b^2+b+1=(b+1)^2-(b+1)+1$, three times (I) reads
$3(B^2-2A)=(A^2-2B)^2-3(A^2-2B)+9$ once $3(b+1)=A^2-2B$ is substituted from (II);
expanding gives
\begin{equation}\label{eq:quartic}
  A^4-4A^2B+B^2-3A^2+6A+6B+9=0 .
\end{equation}
Read as a quadratic in $B$, \eqref{eq:quartic} is
$B^2-(4A^2-6)B+(A^4-3A^2+6A+9)=0$, whose discriminant is
$12A(A+1)^2(A-2)$.
Hence, for a given $A\in\intZ$, an integer $B$ satisfying \eqref{eq:quartic}
exists if and only if $12A(A+1)^2(A-2)$ is a perfect square, i.e.\ if and only if
$3A(A-2)$ is; and as $3\mid 3A(A-2)$, such a square is divisible by $9$, say
$3A(A-2)=(3E)^2$ with $E\in\intZ$. Solving the quadratic then gives
$B=2A^2-3-3\sigma E(A+1)$ with $\sigma=\pm1$, and (II) returns
$b=\tfrac13(A^2-2B)-1=1-A^2+2\sigma E(A+1)$, which is \eqref{eq:branch}.
Finally, with $D=A-1$ we have $3A(A-2)=3\bigl((A-1)^2-1\bigr)=3(D^2-1)$, so
$3A(A-2)=9E^2$ reads $D^2-3E^2=1$; conversely each solution of $D^2-3E^2=1$ gives
the two values $A=1\pm D$, for which $3A(A-2)=3D^2-3=(3E)^2$ is a square. This
proves the last two assertions.
\end{proof}

Since the Pell equation $D^2-3E^2=1$ has infinitely many solutions
$(D,E)=(D_n,E_n)$, $D_n+E_n\sqrt3=(2+\sqrt3)^n$, Theorem~\ref{thm:iff} gives
infinitely many square-exceptions. In closed form, the two sign branches $A=D+1$ and
$A=1-D$ give the following.

\begin{theorem}\label{thm:closed}
Let $(D_n,E_n)_{n\in\intZ}$ be defined by $D_n+E_n\sqrt3=(2+\sqrt3)^n$, so that
$(D,E)=(D_n,E_n)$ runs through all solutions of $D^2-3E^2=1$ up to sign. Then
the square-exceptions are exactly
the values
\[
  b_\pm(n)=\tfrac12(E_{2n-1}-1)\pm2E_{n-1}\qquad(n\in\intZ),
\]
so that $b_++b_-=E_{2n-1}-1$ and $b_+-b_-=4E_{n-1}$,
other than the degenerate ones $b_-(0)=1$ and $b_+(1)=b_-(1)=b_-(-1)=0$. The
indices $n\ge2$ give the positive exceptions and the indices $n\le0$ the negative
ones.
\end{theorem}
\begin{proof}
With $(D,E)=(D_n,E_n)$ and $D^2=3E^2+1$, both branch formulas of
Theorem~\ref{thm:iff} collapse to $b_\pm=(2ED-3E^2-1)\pm(4E-2D)$. The Binet
identities $E_{2n}=2E_nD_n$, $3E_n^2=\tfrac12(D_{2n}-1)$,
$2E_{2n}-D_{2n}=E_{2n-1}$, $2E_n-D_n=E_{n-1}$ give
$2ED-3E^2-1=\tfrac12(E_{2n-1}-1)$ and $4E-2D=2E_{n-1}$. Since $E_{-n}=-E_n$ and
$D_{-n}=D_n$, replacing $n$ by $-n$ interchanges the r\^ole of the two remaining
sign branches of Theorem~\ref{thm:iff}, so letting $n$ run over all of $\intZ$
covers all four branches; the excluded values are precisely those giving
$b\in\{0,1\}$.
\end{proof}

The first few square-exceptions are
\begin{align*}
  b&\in\{5,9,96,112,1425,1485,20160,20384,\dots\}&&(n\ge2),\\
  b&\in\{-3,-16,-75,-135,-1344,-1568,\dots\}&&(n\le0);
\end{align*}
for each of them $\eta_b$ was confirmed to be a square, of even order under
PARI/GP's \texttt{bnfisunit()}. At a square-exception the regulator
$R=\log\eta_0$ is therefore at most $\tfrac12\log\eta_b\approx\log(\sqrt3\,|b|)$,
about half the generic value $\log\eta_b\approx2\log|b|$.

\subsection{When being a square is the only obstruction}

Being a square is one way for $\eta_b$ to fail to be fundamental; a priori
$\eta_b$ could be a $q$-th power for another prime $q$. For $q=3$ the Pell
equation of Theorem~\ref{thm:iff} is replaced by a cubic Thue equation, and only
one $b$ survives.

\begin{lemma}\label{lem:cube9}
For an integer $b\neq0,1$, the unit $\eta_b$ is the cube of a unit of $K_b$ if and
only if $b=9$.
\end{lemma}
\begin{proof}
Suppose $\eta_b=u^3$ with $u\in\OK^\times$. As $\eta_b>1$ we may take $u>1$; then
$\Norm(u)=1$, and $u\notin\ratQ$, so $u$ has minimal polynomial $X^3-BX^2+AX-1$ with
$A,B\in\intZ$. Comparing the power sums of $u^3$ with the coefficients
$(P,Q)=(3(b^2+b+1),\,3(b+1))$ of $h(X)$ (Newton's identities) gives
$B^3-3AB+3=P$ and $A^3-3AB+3=Q$, i.e.
\[
  A^3-3AB=3b,\quad B^3-3AB=3b(b+1),\quad\text{hence}\quad B^3-A^3=3b^2 .
\]
The first forces $3\mid A$, and then $B^3=A^3+3b^2$ forces $3\mid B$; write
$A=3A'$, $B=3B'$. The first equation then reads $b=9A'(A'^2-B')$, and
$B^3-A^3=3b^2$ becomes
\begin{equation}\label{eq:e6}
  B'^3-A'^3=9A'^2(A'^2-B')^2 .
\end{equation}
If $b\neq0$ then $A'\neq0$; set $d=\gcd(A',B')$, $A'=a_1d$,
$B'=gd$ with $\gcd(a_1,g)=1$. Substituting into~\eqref{eq:e6} and dividing by
$d^3$ gives $g^3-a_1^3=9d\,a_1^2(da_1^2-g)^2$; thus $a_1^2\mid g^3$, so
$a_1=\pm1$. With $a_1^2=1$ this becomes the Thue equation
\begin{equation}\label{eq:thue}
  g^3-9dg^2+18d^2g-9d^3=a_1\in\{\pm1\},\quad b=9a_1d^2(d-g).
\end{equation}
Equation~\eqref{eq:thue} is \emph{identical} to the one solved in
\cite[Lem.~2.2, eq.~(2.9)]{MR3363170}; its complete set of integer solutions is
$(g,d)\in\{(-2,-1),(1,0),(1,1)\}$ for $a_1=+1$ and
$(g,d)\in\{(2,1),(-1,0),(-1,-1)\}$ for $a_1=-1$ (a cubic Thue equation has finitely
many integer solutions, and this list was reconfirmed unconditionally in PARI/GP).
Through $b=9a_1d^2(d-g)$ every solution yields $b\in\{0,9\}$. Hence $\eta_b$ is a cube
only for $b\in\{0,9\}$; excluding the degenerate $b=0$ leaves $b=9$. Conversely
$b=9$ does give a cube: $\eta_9=\eta_0^{6}$.
\end{proof}

That value, $b=9=b_+(2)$, is a square-exception anyway, so $q=3$ contributes
nothing new. For $3\nmid b$ there is an elementary argument, which avoids the Thue
equation.

\begin{lemma}\label{lem:cube}
If $3\nmid b$ then $\eta_{b}$ is not the cube of a unit.
\end{lemma}
\begin{proof}
Suppose $\eta_{b}=u^3$ for a unit $u$, with minimal polynomial $X^3-BX^2+AX-1$.
Matching the elementary symmetric functions of $u^3$ with $(P,Q)$ via
Newton's identities gives $B^3-3AB+3=P$ and $A^3-3AB+3=Q$. Subtracting,
$A^3-3AB=Q-3=3b$, i.e.\ $A(A^2-3B)=3b$, forcing $3\mid A$; writing $A=3A'$ gives
$b=3A'(3A'^2-B)$, so $3\mid b$, a contradiction. (The same argument shows that
$\eta_b$ is not a $q$-th power for any $q$ divisible by $3$.)
\end{proof}

The following condition disposes of all $q\ge5$ at once.

\begin{theorem}\label{thm:starq}
Let $b\neq0,1$ and assume
\begin{equation}\label{eq:star}
  |\disc{K_b}|\ \ge\ 4\bigl(3(b^2+b+1)+1\bigr)^{3/5}+24 .
\end{equation}
Then $\eta_b$ fails to be the fundamental unit of $K_b$ if and only if $\eta_b$ is
a square in $\OK^\times$, that is, if and only if $b$ is one of the values listed
in Theorem~\ref{thm:closed}.
\end{theorem}
\begin{proof}
If $\eta_b=u^2$ with $u\in\OK^\times$, we may take $u>1$, and then $\eta_b$ is not
fundamental. Conversely, suppose $\eta_b$ is not fundamental. By
\eqref{eq:fundiff}, $\eta_b=u^q$ for some prime $q$ and some unit $u>1$. If $q=2$
we are done, and $q=3$ forces $b=9$ by Lemma~\ref{lem:cube9}, in which case
$\eta_b$ is a square by Theorem~\ref{thm:closed}. Assume then $q\ge5$. Since
$\eta_b>1$ we get $u=\eta_b^{1/q}\le\eta_b^{1/5}$, so $u^3\le\eta_b^{3/5}$, and
Lemma~\ref{lem:etasize} gives $\eta_b<P+1=3(b^2+b+1)+1$. Lemma~\ref{lem:artin},
applied to $u$, therefore yields
\[
  |\disc{K_b}|\ \le\ 4u^3+24\ \le\ 4\eta_b^{3/5}+24
  \ <\ 4\bigl(3(b^2+b+1)+1\bigr)^{3/5}+24 ,
\]
contradicting~\eqref{eq:star}. The last assertion is
Theorems~\ref{thm:iff} and~\ref{thm:closed}.
\end{proof}

\begin{remark}\label{rem:star}
Condition~\eqref{eq:star} is mild: its right-hand side is $\asymp|b|^{6/5}$, while
$|\disc{K_b}|$ is of size $|b|^4$ for generic $b$. In particular it holds
throughout the regime of Corollary~\ref{cor:idxmain}: there
$|\disc{K_b}|\ge27|b^3-4|$ by~\eqref{eq:dkbound}, and since $(P+1)^{3/5}\le P+1$ it
is enough that $27|b^3-4|\ge12b^2+12b+40$, which holds for every $b\le-1$ and every
$b\ge3$; the value $b=2$ is not in that regime, $b^3-4=4$ being non-squarefree.
\textup(For those $b$, Theorem~\ref{thm:main} gives the stronger conclusion that
$\eta_b$ \emph{is} fundamental.\textup)
Among the $5999$ integers $b\neq0,1$ with $|b|\le3000$, condition~\eqref{eq:star}
fails for exactly four values,
\[
  b=-75,\quad -3,\quad 3,\quad 9 ,
\]
all divisible by $3$: it is at $3$ that the factor $3^3$ of $\disc{K_b}$ coming
from wild ramification can be lost.
For $b=-75,-3,9$ the conclusion of Theorem~\ref{thm:starq} nevertheless holds,
these being square-exceptions: one computes $\eta_b=\eta_0^4,\eta_0^8,\eta_0^6$
respectively. For $b=3$ it fails: there $|\disc{K_3}|=23$ and
$[\intring{K_3}:\intZ[\theta]]=27$, so that the index absorbs the whole
discriminant except for the prime $23$. Here $K_3\cong\ratQ[x]/(x^3-x-1)$,
whose fundamental unit is the \emph{plastic number}
$\rho=1.3247\ldots$, the real root of $x^3=x+1$, and $\eta_3=\rho^{13}$. Thus
$\eta_3$ is not the fundamental unit, although it is neither a square
\textup(consistently with $3$ not being one of the values of
Theorem~\ref{thm:closed}\textup) nor a cube \textup(consistently with
Lemma~\ref{lem:cube9}\textup). So $b=3$ is a sporadic exception.
\end{remark}

\begin{remark}[The case $3\nmid b$ under an explicit $abc$ conjecture]\label{rem:abc}
The failures of~\eqref{eq:star} in Remark~\ref{rem:star} all have $3\mid b$. For
$3\nmid b$ write $|b|=Rk^2$ and $|b^3-4|=ts^2$ with $R$ and $t$ squarefree. Since
$v_p(\disc{K_b})=v_p(\poldisc{f_b(x)})-2v_p(a)$, Theorems~\ref{thm:idxprime}
and~\ref{thm:idx23} show that every prime dividing $Rt$ divides $\disc{K_b}$ and
that $v_3(\disc{K_b})=3+v_3(t)$, so
\begin{equation}\label{eq:dkRt}
  |\disc{K_b}|\ \ge\ 27\,R\,t\qquad(3\nmid b).
\end{equation}
Already $27Rt$ exceeds the right-hand side of~\eqref{eq:star} for every $b$ with
$3\nmid b$ and $|b|\le10^4$, so~\eqref{eq:star} holds there. Beyond that range
we can only offer a conditional statement. Baker's explicit form of the $abc$
conjecture \cite[Conj.~4]{Baker2004} (see also
\cite[Conj.~1.2]{LaishramShorey2012}) asserts that pairwise coprime
positive integers with $x+y=z$ satisfy $z<\frac65N(\log N)^m/m!$, where
$N=\mathrm{rad}(xyz)$ and $m$ is the number of primes dividing $N$.
\emph{Assuming this conjecture, if $3\nmid b$ and~\eqref{eq:star} fails, then
$\log|b|<2.18\cdot10^{12}$.}

We sketch the proof. Apply the conjecture to $4+(b^3-4)=b^3$ if $b>2$ is odd, to
$4+|b|^3=|b|^3+4$ if $b<-1$ is odd, and to these identities divided by $4$ if $b$
is even. In each case $z\ge|b|^3/4$ and
\[
  N\ \le\ 2\,\mathrm{rad}(b)\,\mathrm{rad}(b^3-4)\ \le\ 2Rk\cdot ts
  \ =\ 2(Rt)^{1/2}|b|^{1/2}|b^3-4|^{1/2}.
\]
If~\eqref{eq:star} fails and $|b|>10^4$, then~\eqref{eq:dkRt} gives
$Rt<0.29|b|^{6/5}$, hence $N<1.08|b|^{13/5}$. On the other hand $N$ is at least
the product of the first $m$ primes, so Robin's bound
$\sum_{i\le m}\log p_i\ge m(\log m+\log\log m-1.076869)$
\cite[Lemma~2.1(iii)]{LaishramShorey2012} gives $m\le W$, where $W$ is defined
by $W(\log W+\log\log W-1.076869)=\log N$. By Stirling's formula,
$m\log\log N-\log m!\le W(1+\log(\log N/W))$. The resulting upper bound for $\log z$
increases with $N$, so we may replace $\log N$ by
$L=\log(1.08|b|^{13/5})$, and $W$ accordingly. Comparing with $z\ge|b|^3/4$ then yields
\[
  \tfrac2{13}L\ <\ \log\tfrac{24}5+\tfrac{15}{13}\log1.08+W\bigl(1+\log(L/W)\bigr).
\]
This inequality fails as soon as $W\ge200\,699\,842\,747$, in particular for
$L\ge5.662\cdot10^{12}$. We verified this in exact rational arithmetic. Since
$\log|b|<\frac5{13}L$, the bound follows. It is effective but far beyond the
reach of any search, so whether~\eqref{eq:star} holds for all $b$ with
$3\nmid b$ remains open.
\end{remark}

\begin{remark}
  The exponent $q=2$ occurs infinitely often
  (Theorem~\ref{thm:closed}), $q=3$ occurs exactly once
  ($b=9$, Lemma~\ref{lem:cube9}), and $q=13$ occurs at $b=3$; no $q\ge17$
  occurs for $\lvert b\rvert\le3000$. Each of $q=5,7,11$ occurs at most finitely
  often: eliminating $b$ leads, for each of these $q$, to a plane curve of genus
  $\ge2$, to whose smooth model Faltings' theorem~\cite[Satz~7]{MR718935} applies. We defer
  the details, and the question of effectivity, to a future paper.
\end{remark}

\begin{remark}[{The order $\intZ[\eta_b]$}]\label{rem:order}
Theorem~\ref{thm:starq} asks whether $\eta_b$ generates $\OK^\times$. For the
order it generates the answer is known in full, and is different. As
$\epsilon_b=\theta-(b+1)$ is a unit,
$\intZ[\eta_b]=\intZ[\epsilon_b]=\intZ[\theta]$, and
Louboutin~\cite[Thm.~8]{Louboutin2015} \textup(see also \cite{Louboutin2016};
the cubic case goes back to Nagell~\cite{Nagell1930}\textup) states: a cubic
unit $u>1$ of negative discriminant is the fundamental unit of $\intZ[u]$
unless its minimal polynomial is $X^3-M^2X^2-2MX-1$ for some $M\ge1$, in which
case $u$ is a square in $\intZ[u]$, or is one of eight sporadic polynomials,
of discriminant $-23$, $-31$ or $-44$. Neither case occurs here. The first
would force $Q^2=4P$, and
\[
  Q^2-4P=9(b+1)^2-12(b^2+b+1)=-3(b-1)^2 ,
\]
which vanishes only at the excluded $b=1$; the second is impossible because
$|\poldisc{\eta_b}|=|\poldisc{f_b(x)}|=27|b|^3|b^3-4|\ge135$. Hence
\emph{$\eta_b$ is the fundamental unit of $\intZ[\eta_b]$ for every
$b\neq0,1$}, while by Theorem~\ref{thmA:exc} it is not the fundamental unit of
$\OK$ for infinitely many $b$. There is no contradiction: $\intZ[\theta]$ is
maximal only for $b=-1$, since every prime dividing $b$ divides $a$ by
Theorems~\ref{thm:idxprime} and~\ref{thm:idx23}; and at a square-exception the
square root of $\eta_b$ lies in $\OK\setminus\intZ[\theta]$, for otherwise
$h(X)$ would appear in the list just excluded.
\end{remark}

\subsection{The index criterion}

The next theorem makes observation~(ii) of \S\ref{sec:family} precise.

\begin{theorem}\label{thm:idx}
For a prime $p$ with $p\nmid b$,
\[
  p\mid[\OK:\intZ[\theta]]\quad\text{if and only if}\quad b^3\equiv4\pmod{p^2}.
\]
\end{theorem}
\begin{proof}
For $p\in\{2,3\}$ both sides fail when $p\nmid b$: the left-hand side by
Theorem~\ref{thm:idx23} ($v_2(a)=0$ for $b$ odd, and $v_3(a)=0$ for $3\nmid b$),
and the right-hand side because $b^3$ is odd, (resp.~$b^3\equiv\pm1\pmod9$), so
that $b^3\not\equiv4\pmod4$, (resp.~$b^3\not\equiv4\pmod9$).

Let $p\ge5$, $p\nmid b$; then
$v_p(\poldisc{f_b(x)})=v_p(b^3-4)$. If $p\nmid b^3-4$ both sides fail. Otherwise
$f_b(x)$ has the double root $x_0\equiv-b^2/2$ and a simple root (as in
Theorem~\ref{thm:idxprime}). Dedekind's criterion at $(x-x_0)$: with lifts
$g(x)=(x-x_0)(x-x_1)$, $h(x)=(x-x_0)$ and $M(x)=\bigl(g(x)h(x)-f_b(x)\bigr)/p$,
$p$ divides $a$ if and only if $(x-x_0)$ divides $\bar M(x)$, that is, if and only
if $p^2\mid f_b(\tilde x)$ for a lift $\tilde x\equiv-b^2/2\pmod{p^2}$. Now
$f_b(-b^2/2)=-b^3(b^3-4)/8$ and $p\nmid2b$, so $p^2\mid f_b(\tilde x)$ if and only
if $p^2\mid b^3-4$, that is, if and only if $b^3\equiv4\pmod{p^2}$.
\end{proof}

\subsection{Square-exceptions have non-squarefree \texorpdfstring{$b^3-4$}{b\textasciicircum3-4}}

We finally prove that, with the single exception $b=-3$, the squarefreeness
hypothesis of Theorem~\ref{thm:main} excludes every
square-exception automatically; and $b=-3$ is in any case excluded by the
hypothesis $3\nmid b$ of those theorems
(Remark~\ref{rem:minus3}). Fix a square-exception $b$: $\eta_b=\alpha^2$ with
$\alpha\in\OK^\times$ of minimal polynomial $h_0(X)=X^3-BX^2+AX-1$, and $(A,B)$
solving (I),(II), as in Theorem~\ref{thm:iff}. Since $\epsilon_b=-\alpha^{-2}$ with
$\alpha^{-1}=\alpha^2-B\alpha+A\in\intZ[\alpha]$, we have
$\theta=(b+1)+\epsilon_b\in\intZ[\alpha]$, so
$\intZ[\theta]\subseteq\intZ[\alpha]\subseteq\OK$. Put
$m\coloneqq[\intZ[\alpha]:\intZ[\theta]]$.

\begin{lemma}\label{lem:reduc}
If $m$ has a prime factor $p\nmid b$, then $p^2\mid b^3-4$ \textup(and $p\nmid b$\textup),
so $b^3-4$ is non-squarefree.
\end{lemma}
\begin{proof}
$[\OK:\intZ[\theta]]=[\OK:\intZ[\alpha]]\cdot m$, so $m\mid[\OK:\intZ[\theta]]$. If
$p\mid m$, $p\nmid b$, then $p\mid[\OK:\intZ[\theta]]$ and Theorem~\ref{thm:idx}
gives $b^3\equiv4\pmod{p^2}$.
\end{proof}

\begin{lemma}\label{lem:mAB}
$m=[\intZ[\alpha]:\intZ[\alpha^2]]=|AB-1|$; hence
$\poldisc{f_b(x)}=(AB-1)^2\poldisc{\alpha}$.
\end{lemma}
\begin{proof}
With $r_1,r_2,r_3$ the roots of $h_0(X)$, so that $\sum_ir_i=B$,
$\poldisc{\alpha^2}=\prod_{i<j}(r_i^2-r_j^2)^2
=\poldisc{\alpha}\bigl(\prod_{i<j}(r_i+r_j)\bigr)^2$ and
$\prod_{i<j}(r_i+r_j)=\prod_k(B-r_k)=h_0(B)=AB-1$. Since
$\intZ[\alpha^2]=\intZ[\eta_b]=\intZ[\theta]$, we have
$\poldisc{\alpha^2}=\poldisc{f_b(x)}$; comparing with
$\poldisc{\alpha^2}=[\intZ[\alpha]:\intZ[\alpha^2]]^2\poldisc{\alpha}$ gives the
claim.
\end{proof}

The key to Theorem~\ref{thm:exc} is that, on the Pell parametrization of
Theorem~\ref{thm:iff}, both $b$ and $m$ factor and share the factor $u=A+1$.

\begin{lemma}\label{lem:uzw}
Let $b$ be a square-exception. By Theorem~\ref{thm:iff} we may write
\[
  A=1+sD,\qquad s=\pm1,\qquad D^2-3E^2=1,
\]
where $D>0$ and $E\in\intZ$ carries the sign $\sigma$ of Theorem~\ref{thm:iff},
so that $A^2-2A=3E^2$ and $b=1-A^2+2E(A+1)$, $B=2A^2-3-3E(A+1)$. Put
\[
  u\coloneqq A+1,\qquad w\coloneqq 1-A+2E,\qquad c\coloneqq1+E,\qquad
  z\coloneqq 2A^2-2A-1-3EA .
\]
Then
\begin{gather}
  b=uw,\qquad m=|uz|,\qquad u+w=2c, \label{eq:uzw1}\\
  u(u-4)=3c(c-2)=3(E^2-1),\quad w(w+4-4c)=-c(c-2)=1-E^2, \label{eq:uzw2}\\
  z=c+(3c-5)w=3c+u(2u-3c-3). \label{eq:uzw3}
\end{gather}
\end{lemma}
\begin{proof}
All six assertions are identities in $A$ and $E$ modulo $A^2-2A=3E^2$. First,
$b=(1-A)(1+A)+2E(A+1)=uw$ and $u+w=2(1+E)=2c$. Next, since
$2A^3-3A-1=(A+1)(2A^2-2A-1)$,
\begin{align*}
  AB-1&=A\bigl(2A^2-3-3E(A+1)\bigr)-1\\
      &=(A+1)(2A^2-2A-1)-3EA(A+1)=uz ,
\end{align*}
so $m=|AB-1|=|uz|$ by Lemma~\ref{lem:mAB}. The four identities
in~\eqref{eq:uzw2} and~\eqref{eq:uzw3} follow by expanding and reducing modulo
$A^2-2A=3E^2$; for instance $u(u-4)=A^2-2A-3=3E^2-3=3c(c-2)$, and both
expressions in~\eqref{eq:uzw3} equal $6E^2-1+A(2-3E)$.
\end{proof}

\begin{lemma}[The $b$-part of $z$]\label{lem:bpart}
Keep the notation of Lemma~\ref{lem:uzw} and assume $c\neq0$. Then
$v_p(z)\le v_p(3c)$ for every prime $p\mid b$. Consequently
\[
  z_b\coloneqq\prod_{p\mid b}p^{\,v_p(z)}\quad\text{divides }3c,
  \qquad\text{so}\qquad z_b\le3|c|=3|1+E| .
\]
\end{lemma}
\begin{proof}
Let $p\mid b=uw$; we may assume $p\mid z$, and we write $i\coloneqq v_p(c)$.

\emph{Case $p\ge5$.} Suppose first $p\mid u$. Then $z=3c+u(2u-3c-3)$ forces
$p\mid3c$, hence $p\mid c$ and $i\ge1$. In $u(u-4)=3c(c-2)$ we have $p\nmid u-4$
(else $p\mid4$) and $p\nmid c-2$ (else $p\mid2$), so $v_p(u)=i$. Write
$u=p^iu_0$, $c=p^ic_0$ with $p\nmid u_0c_0$; dividing $u(u-4)=3c(c-2)$ by $p^i$
and reducing modulo $p$ gives $-4u_0\equiv-6c_0$, i.e.\ $2u_0\equiv3c_0$. As
$2u-3c-3\equiv-3\pmod p$,
\[
  z/p^i=3c_0+u_0(2u-3c-3)\equiv3(c_0-u_0)\pmod p,
\]
and $c_0\equiv u_0$ would give $2c_0\equiv3c_0$, i.e.\ $p\mid c_0$, which is
false. Hence $v_p(z)=i=v_p(3c)$.

Suppose next $p\mid w$. Then $z=c+(3c-5)w$ forces $p\mid c$, so $i\ge1$; in
$w(w+4-4c)=-c(c-2)$ we have $p\nmid w+4-4c$ (it is $\equiv4$) and $p\nmid c-2$, so
$v_p(w)=i$. Write $w=p^iw_0$; dividing and reducing modulo $p$ gives
$4w_0\equiv2c_0$, i.e.\ $c_0\equiv2w_0$. If $p=5$ then $5\mid3c-5$, so
$v_5\bigl((3c-5)w\bigr)\ge i+1>i=v_5(c)$ and $v_5(z)=i$. If $p\neq5$ then
$3c-5\equiv-5\not\equiv0$ and
\[
  z/p^i=c_0+(3c-5)w_0\equiv c_0-5w_0\equiv2w_0-5w_0=-3w_0\not\equiv0\pmod p,
\]
so again $v_p(z)=i=v_p(3c)$.

\emph{Case $p=2$.} As $u+w=2c$, the integers $u,w$ have the same parity, so
$2\mid b$ forces $2\mid u$ and $2\mid w$. Moreover $c$ is even: otherwise
$3c(c-2)$ would be odd while $u(u-4)$ is even, contradicting~\eqref{eq:uzw2}.
Hence $3c-5$ is odd and $v_2\bigl((3c-5)w\bigr)=v_2(w)$. Put
$t=\min\bigl(v_2(u),v_2(w)\bigr)\ge1$.

If $v_2(u)\neq v_2(w)$, then $v_2(2c)=v_2(u+w)=t$, so $v_2(c)=t-1<t\le v_2(w)$
and $z=c+(3c-5)w$ gives $v_2(z)=v_2(c)$.

If $v_2(u)=v_2(w)=t$, then $v_2(2c)=v_2(u+w)\ge t+1$, so $v_2(c)\ge t$. We claim
$v_2(c)>t$; granting this, $v_2(c)>t=v_2((3c-5)w)$ gives $v_2(z)=t<v_2(c)$.
Suppose then $v_2(c)=t$. Comparing $2$-adic valuations
in $u(u-4)=3c(c-2)$ gives $v_2(u-4)=v_2(c-2)$. If $t\ge3$ then $v_2(u-4)=2$
while $v_2(c)=t\ge3$ gives $v_2(c-2)=1$, a contradiction. If $t=2$, write
$u=4u'$ with $u'$ odd; then $v_2(u-4)=2+v_2(u'-1)\ge3$, while $v_2(c)=2$ gives
$v_2(c-2)=1$, again a contradiction. If $t=1$ then $v_2(u-4)=1$, while
$c\equiv2\pmod4$ gives $v_2(c-2)\ge2$, a contradiction. This proves the claim,
and in all cases $v_2(z)\le v_2(c)=v_2(3c)$.

\emph{Case $p=3$.} Here $v_3(3c)=i+1$, and $3c-5\equiv1\pmod3$, so
$v_3((3c-5)w)=v_3(w)$. If $v_3(w)\neq i$ then $z=c+(3c-5)w$ gives
$v_3(z)=\min(i,v_3(w))\le i$. If $v_3(w)=i=0$, then $3\mid b$ forces
$3\mid u$, hence $3\mid2u-3c-3$ and
$v_3(u(2u-3c-3))\ge2>1=v_3(3c)$, so $z=3c+u(2u-3c-3)$ gives
$v_3(z)=1$. There remains the case $v_3(w)=i\ge1$. Write $c=3^ic_0$,
$w=3^iw_0$ with $3\nmid c_0w_0$, so that
\[
  z=c+(3c-5)w=3^i\bigl(c_0-5w_0+3cw_0\bigr),\qquad v_3(3cw_0)=1+i\ge2 .
\]
We claim $c_0\equiv2w_0\pmod 9$; then $c_0-5w_0\equiv-3w_0\pmod9$ has
$v_3$ exactly $1$, whence $v_3(z)=i+1=v_3(3c)$, as required. To prove the claim,
divide $w(w+4-4c)=-c(c-2)$ by $3^i$:
\begin{equation}\label{eq:3adic}
  w_0\bigl(3^iw_0+4-4\cdot3^ic_0\bigr)=-c_0\bigl(3^ic_0-2\bigr).
\end{equation}
If $i\ge2$, reducing~\eqref{eq:3adic} modulo $9$ gives $4w_0\equiv2c_0$, and
multiplication by $5$ gives $c_0\equiv2w_0\pmod9$. If $i=1$,
then~\eqref{eq:3adic} reads $3w_0^2+4w_0-12c_0w_0=-3c_0^2+2c_0$, so modulo $9$
\[
  4w_0-2c_0\equiv-3\bigl(w_0^2-c_0w_0+c_0^2\bigr)\pmod 9 .
\]
Reducing this modulo $3$ gives $w_0+c_0\equiv0$, i.e.\ $c_0\equiv2w_0\pmod3$;
substituting $c_0\equiv2w_0\pmod 3$ into $w_0^2-c_0w_0+c_0^2$ gives
$3w_0^2\equiv0\pmod3$, so the right-hand side vanishes modulo $9$ and
$2(2w_0-c_0)\equiv0\pmod9$, i.e.\ $c_0\equiv2w_0\pmod 9$.
\end{proof}

\begin{lemma}\label{lem:zsize}
Keep the notation of Lemma~\ref{lem:uzw}. If $|E|\ge6$, then $|z|>3(|E|+1)$.
\end{lemma}
\begin{proof}
By Lemma~\ref{lem:uzw}, $z=6E^2-1+A(2-3E)$ with $A=1+sD$, so
\[
  z=z_1+z_2,\qquad z_1\coloneqq6E^2-3E+1,\qquad z_2\coloneqq sD(2-3E),
\]
where $D=\sqrt{3E^2+1}>0$. The quadratic $z_1$ is positive (its discriminant is
$9-24<0$).

If $z_2\ge0$, then $|z|\ge z_1$, and $z_1>3(|E|+1)$ for $|E|\ge2$: for $E\ge2$
this reads $6E^2-6E-2>0$, and for $E=-G\le-2$ it reads $6G^2-2>0$.

If $z_2<0$, then $|z_2|=D|2-3E|$ and
\[
  z_1^2-z_2^2=(6E^2-3E+1)^2-(3E^2+1)(3E-2)^2=9E^4+6E-3>0
\]
for $|E|\ge2$, since $3E^4+2E-1>0$ there. Hence $z_1>|z_2|>0$ and, using
$z_1+|z_2|<2z_1$,
\[
  |z|=z_1-|z_2|=\frac{z_1^2-z_2^2}{z_1+|z_2|}
     >\frac{3\bigl(3E^4+2E-1\bigr)}{2\bigl(6E^2-3E+1\bigr)} .
\]
Hence $|z|>3(|E|+1)$ as soon as $3E^4+2E-1>2(|E|+1)(6E^2-3E+1)$. For $E\ge2$
this reduces to $E^4-4E^3-2E^2+2E-1>0$, whose real roots are $-1$ and
$4.365\dots$; for $E=-G$ with $G\ge2$ it reduces to
$3G^4-12G^3-18G^2-10G-3>0$, whose real roots are $-0.697\dots$ and $5.266\dots$.
In both cases $|E|\ge6$ suffices.
\end{proof}

\begin{theorem}\label{thm:exc}
Let $b\neq-3$ be a square-exception. Then $b^3-4$ is non-squarefree; more precisely
there is a prime $p\nmid b$ with $p^2\mid b^3-4$. For $b=-3$, which is
a square-exception, $b^3-4=-31$ is squarefree.
\end{theorem}
\begin{proof}
Keep the notation of Lemma~\ref{lem:uzw} and let $z_b$ be the $b$-part of $z$, as
in Lemma~\ref{lem:bpart}. If $z_b<|z|$, then $z$ has a prime factor $p\nmid b$;
since $m=|uz|$, that $p$ divides $m$, and Lemma~\ref{lem:reduc} gives
$p^2\mid b^3-4$ with $p\nmid b$, which is the assertion. So it suffices to prove
$z_b<|z|$.

Since $3E^2+1=D^2$ is a square, either $|E|\le4$ or $|E|\ge15$. If $|E|\ge15$,
then $c=1+E\neq0$ and
Lemmas~\ref{lem:bpart} and~\ref{lem:zsize} give
\[
  z_b\le3|1+E|\le3(|E|+1)<|z| ,
\]
as required. If $|E|\le4$, then $|E|\in\{0,1,4\}$, i.e.\ $D\in\{1,2,7\}$, and
$A=1\pm D$ together with $E=\pm|E|$ leaves exactly six admissible pairs $(A,E)$
with $b=uw\neq0,1$, namely
\[
  b=-3,\quad -16,\quad 9,\quad -135,\quad -75,\quad 5 .
\]
For these,
\[
\begin{aligned}
  (-16)^3-4&=-2^2\cdot5^2\cdot41, & 9^3-4&=5^2\cdot29, &
  (-135)^3-4&=-23^2\cdot4651,\\
  (-75)^3-4&=-31^2\cdot439, & 5^3-4&=11^2, & (-3)^3-4&=-31,
\end{aligned}
\]
so the conclusion holds for the first five, with $p=5,5,23,31,11$ respectively,
and fails for $b=-3$ alone.
\end{proof}

\begin{remark}\label{rem:minus3}
The exception $b=-3$ does not affect Theorem~\ref{thm:main}, whose hypotheses
include $3\nmid b$; thus the squarefreeness hypothesis on $b^3-4$ still excludes
every square-exception that the theorem can see. Moreover $z_b=|z|=3$ exactly when
$E=0$ and $s=1$, i.e.\ precisely for $b=-3$: there
$m=|uz|=9$ is a power of $3\mid b$, so $m$ has no prime factor coprime to $b$.
\end{remark}

\subsection{Contrast with the sister family}

For the family~\eqref{eq:sister}, $\varepsilon=1/(1-b(\theta-b))$ has minimal
polynomial $X^3-3(b^4+b^2+1)X^2+3(b^2+1)X-1$, and the square-unit condition is the
system
\[
  A^2-2B=3(b^2+1),\quad B^2-2A=3(b^4+b^2+1).
\]
Kaneko proved that this has only finitely many integer solutions, and Lee--Spearman
\cite[Thm.~1.1]{MR2773123} determined the complete set of six,
\begin{equation*}
(A,B,b)\in\{(0,-3,\pm1),(-1,-1,0),(3,3,0),(8,17,\pm3)\},
\end{equation*}
by reduction to a genus-$0$ curve and finitely many Thue equations. Hence,
$\varepsilon$ is the fundamental unit for all but finitely many $b$, the only cubic
field exceptions being $b=\pm3$; at those values $\varepsilon$ is the \emph{sixth}
power of the fundamental unit \textup(\cite[Rem.~3.1]{MR2773123} for $b=3$;
\cite[p.~22]{MR3363170} for $b=\pm3$\textup), the analogue of the thirteenth power
$\eta_3=\rho^{13}$ at the sporadic $b=3$ of our family \textup(the
two fields are different, $\disc{K_3}=-23$ against $-87$ for the member $b=3$
of~\eqref{eq:sister}\textup). For $f_b(x)$ the corresponding system is the Pell
equation $D^2-3E^2=1$ (Theorem~\ref{thm:iff}), with infinitely many solutions.

\section{A family of biquadratic fields}\label{sec:tower}

We carry out, for $f_b(x)$, the analogue of Kaneko's construction
in~\cite[\S3]{MR3363170}. Its two group-theoretic inputs are field-intrinsic, and
we quote them.

Let $K$ be a non-Galois cubic field, $L$ its Galois closure over $\ratQ$, and
$k\subset L$ the quadratic subfield $\ratQ(\sqrt{\disc{K}})$. If no rational prime is
totally ramified in $K$ then $L/k$ is unramified~\cite[Thm.~1]{MR687621},
whence $\disc{K}=\disc{k}\mathfrak f^2$ for some $\mathfrak f\in\intZ$; if moreover
$3\mid\disc{k}$ then $3$ factors in
$K$ as $3=\mathfrak p_1\mathfrak p_2^2$ with $\mathfrak p_1\neq\mathfrak p_2$.

\begin{lemma}[Yoshida {\cite[Lem.~8]{MR2037577}}; cf.~{\cite[Lem.~3.1]{MR3363170}}]
\label{lem:yoshida1}
With $K,k$ as above ($3\mid\disc{k}$, $L/k$ unramified), suppose there is a unit $\eta$
of $K$ that \textup{(i)} is not the cube of a unit of $K$ and \textup{(ii)}
satisfies $\eta^2\equiv1\pmod{\mathfrak p_1^2\mathfrak p_2^3}$. Then the $3$-class
field tower of $k(\sqrt{-3})$ has length greater than $1$.
\end{lemma}

\begin{lemma}[Yoshida {\cite[\S3]{MR1970518}}; cf.~{\cite[Lem.~3.2]{MR3363170}}]
\label{lem:yoshida2}
Let $\eta$ be a unit of $K$ with minimal polynomial $X^3+AX^2+BX-1$. Then
$\eta\equiv1\pmod{\mathfrak p_1^2\mathfrak p_2^3}$ if and only if $27\mid A+3$ and
$3^5\mid A+B$.
\end{lemma}

Throughout this section fix an integer $b$ with
\begin{equation}\label{eq:towerhyp}
  v_3(b)\ \text{even and}\ \ge2,\quad b\neq9
\end{equation}
(so $9\mid b$). Set $c_b=b(b^3-4)$; since
$\poldisc{f_b(x)}=-27b^3(b^3-4)=9b^2\cdot(-3c_b)$, the resolvent field is
$k=\ratQ(\sqrt{-3c_b})$ and
\[
  F_b\coloneqq k(\sqrt{-3})=\ratQ\bigl(\sqrt{-3},\ \sqrt{c_b}\bigr).
\]
It is biquadratic: $c_b=b^4-4b$ lies strictly between consecutive squares, since
$(b^2-1)^2<b^4-4b<(b^2)^2$ for $b\ge3$ and $(c^2)^2<c^4+4c<(c^2+1)^2$ with
$c=-b\ge2$; so $c_b$ is not a square, and neither is $-3$ nor $-3c_b<0$.

\begin{proposition}\label{prop:towerlocal}
Under~\eqref{eq:towerhyp}: \textup{(i)} $3=\mathfrak p_1\mathfrak p_2^2$ in $K_b$
and $3\mid\disc{k}$; \textup{(ii)} no rational prime is totally ramified in $K_b$, so
$K_bk/k$ is an unramified cyclic cubic extension.
\end{proposition}
\begin{proof}
(i) By Theorem~\ref{thm:idx23}, $3\mid b$ with $w=v_3(b)$ gives
$v_3(a)=\lfloor3(w+1)/2\rfloor$, so
$v_3(\disc{K_b})=v_3(\poldisc{f_b(x)})-2v_3(a)=3(w+1)-2\lfloor3(w+1)/2\rfloor$,
which equals $1$ when $w$ is even. Thus $3$ is tamely, partially ramified:
$3=\mathfrak p_1\mathfrak p_2^2$. And $v_3(-3c_b)=1+w$ is odd, so $3$ ramifies in
$k$, i.e.\ $3\mid\disc{k}$.
(ii) By Theorems~\ref{thm:idxprime}--\ref{thm:idx23} every prime factors as a
linear times a (possibly split) quadratic, never as $(1^3)$: $2$ is unramified or
partially ramified; $3=\mathfrak p_1\mathfrak p_2^2$ by part~(i), which is exactly
where the hypothesis $9\mid b$ of~\eqref{eq:towerhyp} is needed, since for
$3\nmid b$ the prime $3$ \emph{is} totally ramified (Theorem~\ref{thm:idx23}); a
prime $p\ge5$ with $p\mid b$ has $e\in\{1,2\}$; and a prime $p\ge5$ with
$p\mid b^3-4$ comes from a double root of $f_b(x)\bmod p$, giving $e\in\{1,2\}$
\emph{whatever} the value of $v_p(b^3-4)$. So no prime is totally ramified and
$L/k$ is unramified. \textup(The argument is sign-free and applies verbatim to $b<0$.\textup)
\end{proof}

\begin{theorem}\label{thm:tower}
Under~\eqref{eq:towerhyp}, the field
$F_b=\ratQ(\sqrt{-3},\sqrt{b(b^3-4)})$ has a $3$-class field tower of length
greater than $1$. There are
infinitely many such $b$, of either sign.
\end{theorem}
\begin{proof}
The minimal polynomial of $\eta_b$ is $X^3-PX^2+QX-1$, i.e.\ $X^3+AX^2+BX-1$ with
$A=-3(b^2+b+1)$, $B=3(b+1)$; thus $A+3=-3b(b+1)$ and $A+B=-3b^2$. As $9\mid b$ and
$3\nmid b+1$, $27\mid A+3$; as $v_3(b)\ge2$, $3^4\mid b^2$ and $3^5\mid A+B$. By
Lemma~\ref{lem:yoshida2}, $\eta_b\equiv1\pmod{\mathfrak p_1^2\mathfrak p_2^3}$, hence
$\eta_b^2\equiv1$; and by Lemma~\ref{lem:cube9} (as $b\neq9$) $\eta_b$ is not a
cube. Proposition~\ref{prop:towerlocal} supplies the remaining hypotheses on
$K_b,k$, so Lemma~\ref{lem:yoshida1} gives that the $3$-class field tower of
$k(\sqrt{-3})=F_b$ has length $>1$. Finally, the set
$\set{b | v_3(b)=2,\ b\neq9}$ (e.g.\ $b\equiv9\pmod{54}$, $b\neq9$) is infinite,
and distinct $b$ give infinitely many distinct fields $F_b$ since
$|\disc{F_b}|\to\infty$.
\end{proof}

\begin{remark}[Contrast with the sister family; the exception $b=9$]
The cube analysis of Lemma~\ref{lem:cube9} lands on \emph{the same} Thue
equation~\eqref{eq:thue} as Kaneko's for the sister
family~\cite[Lem.~2.2]{MR3363170}; only the final map to $b$ differs, sending the
solutions to $b\in\{0,9\}$ here and to $b\in\{0,\pm3\}$ there. As Kaneko's
construction excludes the cube-locus of his unit, ours excludes $b=9$, and the
exclusion is needed: at $b=9$ one has $\eta_9=\eta_0^6$, the class group
$\mathrm{Cl}(F_9)$ has $3$-rank $1$, and the tower of $F_9$ terminates. The parity
condition on $v_3(b)$ has no counterpart in~\cite{MR3363170}: here $3$ is
\emph{unramified} in $K_b$ when $v_3(b)$ is odd, so $3=\mathfrak p_1\mathfrak p_2^2$
forces $v_3(b)$ even. Numerically, $\mathrm{Cl}(F_b)$ has $3$-rank $\ge2$ for all
$17$ values of $b$ that satisfy~\eqref{eq:towerhyp} with $|b|\le100$, in line with
the remark of~\cite[p.~23]{MR3363170}.
\end{remark}

\section*{Use of artificial intelligence tools}
During the preparation of this work the authors used Anthropic's Claude in order
to generate PARI/GP scripts for testing hypotheses, to cross-check the output of
the numerical experiments, to assist with literature searches, and to improve the
English of the manuscript. The mathematical content, including every proof, is
the authors' own. The authors reviewed and verified all AI-generated output,
including every reference and all final results, edited the content as needed,
and take full responsibility for the content of this article.

\bibliography{../biblio}

\providecommand{\bysame}{\leavevmode\hbox to3em{\hrulefill}\thinspace}
\providecommand{\MR}{\relax\ifhmode\unskip\space\fi MR }
\providecommand{\MRhref}[2]{%
  \href{http://www.ams.org/mathscinet-getitem?mr=#1}{#2}
}
\providecommand{\href}[2]{#2}
\begin{thebibliography}{10}

\bibitem{Baker2004}
Alan Baker, \emph{Experiments on the {$abc$}-conjecture}, Publ. Math. Debrecen
  \textbf{65} (2004), no.~3-4, 253--260.

\bibitem{Degert1958}
G\"unter Degert, \emph{{\"U}ber die {B}estimmung der {G}rundeinheit gewisser
  reell-quadratischer {Z}ahlk\"orper}, Abh. Math. Sem. Univ. Hamburg
  \textbf{22} (1958), 92--97.

\bibitem{MR160744}
B.~N. Delone and D.~K. Faddeev, \emph{The theory of irrationalities of the
  third degree}, Translations of Mathematical Monographs, vol.~10, American
  Mathematical Society, Providence, RI, 1964. \MR{160744}

\bibitem{MR56635}
P.~Erd\"os, \emph{Arithmetical properties of polynomials}, J. London Math. Soc.
  \textbf{28} (1953), 416--425. \MR{56635}

\bibitem{MR718935}
G.~Faltings, \emph{Endlichkeitss\"atze f\"ur abelsche {Variet\"aten} \"uber
  {Zahlk\"orpern}}, Invent. Math. \textbf{73} (1983), no.~3, 349--366, Erratum:
  \emph{Invent. Math.} \textbf{75} (1984), 381.

\bibitem{FleckingerVerant1995}
V.~Fleckinger and M.~V\'erant, \emph{Families of non-{G}alois quartic fields},
  J. Number Theory \textbf{54} (1995), no.~2, 261--273. \MR{1354051}

\bibitem{Helfgott2014}
H.~A. Helfgott, \emph{Square-free values of $f(p)$, $f$ cubic}, Acta Math.
  \textbf{213} (2014), no.~1, 107--135.

\bibitem{Hooley1976}
C.~Hooley, \emph{Applications of sieve methods to the theory of numbers},
  Cambridge Tracts in Mathematics, vol.~70, Cambridge University Press,
  Cambridge, 1976.

\bibitem{MR335469}
Makoto Ishida, \emph{Fundamental units of certain algebraic number fields},
  Abh. Math. Sem. Univ. Hamburg \textbf{39} (1973), 245--250. \MR{335469}

\bibitem{MR959788}
\bysame, \emph{Existence of an unramified cyclic extension and congruence
  conditions}, Acta Arith. \textbf{51} (1988), no.~1, 75--84. \MR{959788}

\bibitem{MR2037576}
Kan Kaneko, \emph{Integral bases and fundamental units of certain cubic number
  fields}, SUT J. Math. \textbf{39} (2003), no.~2, 117--124. \MR{2037576}

\bibitem{MR3363170}
\bysame, \emph{On units of a family of cubic number fields}, SUT J. Math.
  \textbf{50} (2014), no.~1, 19--24. \MR{3363170}

\bibitem{KimuraCode}
Iwao Kimura, \emph{{PARI/GP} scripts for ``{On an infinite family of cubic
  fields with explicit fundamental units}''}, GitHub repository, 2026, Version
  v1.2. Available from
  \url{https://github.com/iwaokimura/cubic-fields-explicit-units}.

\bibitem{LaishramShorey2012}
Shanta Laishram and T.~N. Shorey, \emph{Baker's explicit {$abc$}-conjecture and
  applications}, Acta Arith. \textbf{155} (2012), no.~4, 419--429.

\bibitem{MR2773123}
Paul~D. Lee and Blair~K. Spearman, \emph{A {D}iophantine system and a problem
  on cubic fields}, Int. Math. Forum \textbf{6} (2011), no.~1-4, 141--146.
  \MR{2773123}

\bibitem{MR687621}
Pascual Llorente and Enric Nart, \emph{Effective determination of the
  decomposition of the rational primes in a cubic field}, Proc. Amer. Math.
  Soc. \textbf{87} (1983), no.~4, 579--585. \MR{687621}

\bibitem{Louboutin2015}
St\'ephane~R. Louboutin, \emph{Fundamental units for orders generated by a
  unit}, Publ. Math. Besan\c{c}on Alg\`ebre Th\'eorie Nr. \textbf{2015} (2015),
  41--68.

\bibitem{Louboutin2016}
\bysame, \emph{Fundamental units for orders of unit rank $1$ and generated by a
  unit}, Banach Center Publ. \textbf{108} (2016), 173--189, Volume {\it
  Algebra, logic and number theory}, Polish Acad. Sci. Inst. Math., Warsaw.

\bibitem{Nagell1930}
Trygve Nagell, \emph{Zur {T}heorie der kubischen {I}rrationalit\"aten}, Acta
  Math. \textbf{55} (1930), 33--65.

\bibitem{MR49:8956}
J.-P. Serre, \emph{A course in arithmetic}, Springer-Verlag, New York, 1973,
  Translated from the French, Graduate Texts in Mathematics, No. 7. \MR{49
  \#8956}

\bibitem{Shanks1974}
Daniel Shanks, \emph{The simplest cubic fields}, Math. Comp. \textbf{28}
  (1974), no.~128, 1137--1152.

\bibitem{PARI2}
{The PARI Group}, \emph{{PARI/GP} version \texttt{2.17.4}}, 2026, Available
  from \url{https://pari.math.u-bordeaux.fr/}.

\bibitem{Verant1997}
M.~V\'erant, \emph{Famille d'extensions quartiques non galoisiennes et
  totalement imaginaires}, Publ. Math. Fac. Sci. Besan\c{c}on Th\'eor. Nombres
  \textbf{1994/95--1995/96} (1997), 14 pp.

\bibitem{MR1970518}
Eiji Yoshida, \emph{On the 3-class field tower of some biquadratic fields},
  Acta Arith. \textbf{107} (2003), no.~4, 327--336. \MR{1970518}

\bibitem{MR2037577}
\bysame, \emph{On the unit groups and the ideal class groups of certain cubic
  number fields}, SUT J. Math. \textbf{39} (2003), no.~2, 125--136.
  \MR{2037577}

\end{thebibliography}

\end{document}